\documentclass[11pt, letterpaper]{amsart}

\usepackage{amssymb,euscript,tikz,units}
\usepackage[colorlinks,citecolor=blue,linkcolor=red]{hyperref}
\usepackage{verbatim}
\usepackage[nameinlink]{cleveref}
\usepackage{enumitem}
\usepackage{hyperref}
\usepackage{url}

\newcommand{\M}{\mathcal{M}}
\newcommand{\sM}{\mathcal{M}}

\newcommand{\sL}{\mathcal{L}}

\newcommand{\sY}{\mathcal{Y}}
\newcommand{\R}{\mathbb{R}}
\newcommand{\Z}{\mathbb{Z}}

\newcommand{\E}{\mathbb{E}_{\rm WP}^g}

\newcommand{\eg}{\textit{e.g.\@ }}

\newcommand{\ie}{\textit{i.e.\@}}

\def\sys{\mathop{\rm sys}}
\def\area{\mathop{\rm Area}}
\def\arcsinh{\mathop{\rm arcsinh}}

\def\Vol{\mathop{\rm Vol}}

\def\dist{\mathop{\rm dist}}
\def\inj{\mathop{\rm inj}}
\def\Prob{\mathop{\rm Prob}\nolimits_{\rm WP}^g}

\DeclareMathOperator{\diam}{diam}

\theoremstyle{plain}
\newtheorem{theorem}{Theorem}

\newtheorem{proposition}[theorem]{Proposition}
\newtheorem{lemma}[theorem]{Lemma}

\newtheorem{remark}[theorem]{Remark}

\newtheorem{example}[theorem]{Example}

\newcommand{\be}{\begin{equation}}
\newcommand{\ene}{\end{equation}}
\newcommand{\br}{\begin{remark}}
\newcommand{\er}{\end{remark}}
\newcommand{\bl}{\begin{lemma}}
\newcommand{\el}{\end{lemma}}
\newcommand{\bcor}{\begin{cor}}
\newcommand{\ecor}{\end{cor}}
\newcommand{\bpro}{\begin{pro}}
\newcommand{\epro}{\end{pro}}
\newcommand{\ben}{\begin{enumerate}}
\newcommand{\een}{\end{enumerate}}
\newcommand{\bp}{\begin{proof}}
\newcommand{\ep}{\end{proof}}
\newcommand{\bpo}{\begin{pro}}
\newcommand{\epo}{\end{pro}}
\newcommand{\beq}{\begin{equation*}}
\newcommand{\eeq}{\end{equation*}}
\newcommand{\bear}{\begin{eqnarray}}
\newcommand{\eear}{\end{eqnarray}}
\newcommand{\beqar}{\begin{eqnarray*}}
\newcommand{\eeqar}{\end{eqnarray*}}
\newcommand{\bt}{\begin{theorem}}
\newcommand{\et}{\end{theorem}}
\newcommand{\bex}{\begin{excer}}
\newcommand{\eex}{\end{excer}}

\theoremstyle{definition}

\theoremstyle{remark}

\newtheorem*{rem*}{Remark}
\newtheorem*{ques*}{Question}
\newtheorem*{def*}{Definition}
\newtheorem*{con*}{Construction}
\newtheorem*{thm*}{\bf Theorem}
\newtheorem*{definition*}{Definition}
\newtheorem*{assum*}{Assumption $(\star)$}
\newtheorem*{obs*}{Observation}

\usepackage[utf8]{inputenc}
\usepackage[justification=centering]{caption}

\begin{document}

\title[Second eigenvalues for large genus]{On second eigenvalues of closed hyperbolic surfaces for large genus}
\author{Yuxin He and Yunhui Wu}
\address{Yau Mathematical Sciences Center \& Department of Mathematical Sciences, Tsinghua University, Haidian District, Beijing 100084, China}
\email[(Y.~H.)]{hyx21@mails.tsinghua.edu.cn}
\email[(Y.~W.)]{yunhui\_wu@tsinghua.edu.cn}

\maketitle

\begin{abstract}
In this article, we study the second eigenvalues of closed hyperbolic surfaces for large genus. We show that for every closed hyperbolic surface $X_g$ of genus $g$ $(g\geq 3)$, up to uniform positive constants multiplications, the second eigenvalue $\lambda_2(X_g)$ of $X_g$ is greater than $\frac{\mathcal{L}_2(X_g)}{g^2}$ and less than $\mathcal{L}_2(X_g)$; moreover these two bounds are optimal as $g\to \infty$. Here $\mathcal{L}_2(X_g)$ is the shortest length of simple closed multi-geodesics separating $X_g$ into three components. 

Furthermore, we also investigate the quantity $\frac{\lambda_2(X_g)}{\mathcal{L}_2(X_g)}$ for random hyperbolic surfaces of large genus. We show that as $g\to \infty$, a generic hyperbolic surface $X_g$ has  $\frac{\lambda_2(X_g)}{\mathcal{L}_2(X_g)}$  uniformly comparable to $\frac{1}{\ln(g)}$.
\end{abstract}

\section{Introduction}\label{introduc}

The spectrum of the Laplacian on a hyperbolic surface has been a fascinating topic in several mathematical fields including analysis, dynamics, geometry, mathematical physics, number theory and so on for a long time. Let $X_g$ be a closed hyperbolic surface of genus $g$. The spectrum of $X_g$ is a discrete closed subset in $\R^{\geq 0}$ and consists of eigenvalues with finite multiplicity. We enumerate them, counted with multiplicity, in the following increasing order:
\[0=\lambda_0(X_g)<\lambda_1(X_g)\leq \lambda_2(X_g) \leq \cdots \to \infty.\]

\noindent It is known by Buser \cite{Bus77} that the $(2g-3)$-th eigenvalue can be arbitrarily closed to $0$; and the $n$-th eigenvalue can be arbitrarily closed to $\frac{1}{4}$ for any $n\geq (2g-2)$. Otal-Rosas \cite{OR09} showed that the $(2g-2)$-th eigenvalue is always greater than $\frac{1}{4}$. One may also refer to Ballmann-Matthiesen-Mondal \cite{BMM16,BMM17} and Mondal \cite{Mon14} for more recent general statements on $\lambda_{2g-2}(X_g)$.

\begin{definition*}
For a hyperbolic surface $X$ and any integer $k\in [1,|\chi(X)|-1]$ where $\chi(X)$ is the Euler characteristic of $X$, we define a positive quantity $\sL_k(X)$ of $X$ to be the minimal possible sum of the lengths of simple closed multi-geodesics in $X$ that cut $X$ into $k+1$ components.
\end{definition*}

\noindent In the closed case, the quantity $\sL_k(X_g)$ can be arbitrarily closed to $0$ when $X_g$ is close enough to a maximal cusped surface. Schoen-Wolpert-Yau \cite{SWY80} showed that for any integer $k \in [1, 2g-3]$, there exist two constants $\alpha_k(g)>0$ and $\beta_k(g)>0$, depending on $k$ and $g$, such that \bear \label{SWY-eigen} \alpha_k(g)\leq \frac{\lambda_k(X_g)}{\sL_k(X_g)}\leq \beta_k(g).\eear
One may see Dodziuk-Randol \cite{DR86} for a different proof, and also see Dodziuk-Pignataro-Randol-Sullivan \cite{DPRS85} for similar results of cusped hyperbolic surfaces. Motivated by the work of Dodziuk-Randol \cite{DR86}, joint with Xue the second named author in \cite{wu2021optimal} studied the asymptotic behavior of the constant $\alpha_1(g)$ for large genus and showed that for any closed hyperbolic surface $X_g$ of genus $g$, 
\[\lambda_1(X_g)\succ \frac{\mathcal{L}_1(X_g)}{g^2}.\]
Moreover, it is known by their earlier work \cite{WX18} that for all $g\geq 2$, there exists a closed hyperbolic surface $\mathcal{Z}_g$ of genus $g$ such that $\lambda_1(\mathcal{Z}_g)\asymp \frac{\mathcal{L}_1(\mathcal{Z}_g)}{g^2}$. 
And they pose the following question:
\begin{ques*}\label{ulb-eig-lg}
For every $g\geq 2$, $k\in [1,2g-3]$ and any closed hyperbolic surface $X_g$ of genus $g$, 
\[\lambda_k(X_g)\succ \frac{k\cdot \mathcal{L}_k(X_g)}{g^2}?\]
\end{ques*}
\noindent As introduced above the case of $k=1$ was positively answered in \cite{wu2021optimal}. In this paper we prove
\bt \label{mt-1}
For every $g\geq 3$ and any closed hyperbolic surface $X_g$ of genus $g$,  
\[\frac{\mathcal{L}_2(X_g)}{g^2}\prec \lambda_2(X_g) \prec \mathcal{L}_2(X_g).\]
Moreover, these two bounds above are optimal for large genus in the sense that for all $g\geq 3$, there exist two closed hyperbolic surfaces $\mathcal{X}_g$ and $\mathcal{Y}_g$ of genus $g$ such that $$\lambda_2(\mathcal{X}_g) \asymp \frac{\mathcal{L}_2(\mathcal{X}_g)}{g^2}\ \textit{and} \ \lambda_2(\mathcal{Y}_g) \asymp\mathcal{L}_2(\mathcal{Y}_g).$$
\et

\begin{rem*}
\ben
\item The proof of Theorem \ref{mt-1} above also works for the general index $k$ independent of $g$, which we leave to interested readers. For simplicity, we only argue the case of $k=2$ in this paper. So the remaining \emph{unsolved} case of the question above is for the index $k=k(g) \in [1, 2g-3]$ depending on $g$.
\item The upper bound in Theorem \ref{mt-1} follows from a refined argument in \cite{SWY80}. However, the existence of $\mathcal{Y}_g$ with $\lambda_2(\mathcal{Y}_g) \asymp \mathcal{L}_2(\mathcal{Y}_g)$ is nowhere trivial, which relies on robust results in \cite{mirzakhani2013growth} of Mirzakhani on Weil-Petersson random hyperbolic surfaces of large genus.
\een
\end{rem*}

For the proof of the lower bound in Theorem \ref{mt-1}, the following result on the first positive Neumann eigenvalue plays a key role, which is also of independent interest. More precisely,
\begin{theorem}\label{mt-2}
If $X_{g,n}$ is a compact hyperbolic surface of genus $g$ with $n$ geodesic boundary component, where $2g+n\geq 4$, and the length of each boundary component is less than or equal to $L$ for some constant $L>0$, then  there exists a constant $C(L)>0$  only depending on $L$ such that $$
\sigma_1(X_{g,n})\geq C(L)\cdot \frac{\mathcal{L}_1(X_{g,n})}{(2g-2+n)^2}
$$
where $\sigma_1(X_{g,n})$ is the first positive Neumann eigenvalue of $X_{g,n}$.
\end{theorem}

\noindent We remark here that the constant $C(L)>0$ in the theorem above can not be chosen to be uniform as in \cite{wu2021optimal} for closed hyperbolic surface case: actually we will construct an example to see that $C(L)\to 0$ as $L=L(g)\to \infty$ (see Example \ref{Examp}). In either \cite{SWY80} or \cite{DR86}, the lower bounds in \eqref{SWY-eigen} for index $i\in [2,2g-3]$ follow from $\lambda_1(X_g)\geq \alpha_1(g) \cdot \sL_1(X_g)$ together with a max-minimal  principle. However, when applying the same argument together by using $\lambda_1(X_g)\succ \frac{\mathcal{L}_1(X_g)}{g^2} $ in \cite{wu2021optimal}, one can not get the lower bound in Theorem \ref{mt-1} because $\sL_1(X_g)$ may go to infinity as $g\to \infty$, preventing us from directly applying Theorem \ref{mt-2}. The major contribution for the lower bound in Theorem \ref{mt-1} is to overcome this difficulty and get the optimal coefficient $\frac{1}{g^2}$ for large $g$.\\

Let $\sM_g$ denote the moduli space of closed hyperbolic surfaces of genus $g$ endowed with the Weil-Petersson metric. Recently the study of random surfaces in $\sM_g$ for large genus is quite active, which was initiated by Mirzakhani \cite{mirzakhani2013growth} based on her celebrated thesis works \cite{Mirz07,Mirz07-int}. One may see \cite{GPY11, MP19, MT20, NWX20, PWX21} and references therein for recent developments on the geometry of random surfaces of large genus; and see \cite{GMST21, MS20, Thomas22, Monk21, WX21, LW21, Hide21, Ru22, SW22} and references therein for recent developments on the spectral theory of random surfaces of large genus. 

We view the quantity $\frac{\lambda_2(X_g)}{\sL_2(X_g)}$ as a positive random variable on $\sM_g$. Theorem \ref{mt-1} tells that
\[\inf\limits_{X_g\in \sM_g}\frac{\lambda_2(X_g)}{\sL_2(X_g)}\asymp \frac{1}{g^2} \ \textit{and} \ \sup\limits_{X_g\in \sM_g}\frac{\lambda_2(X_g)}{\sL_2(X_g)}\asymp 1.\]
In this paper, we show that as $g\to \infty$, a generic $X_g\in \sM_g$ has  $\frac{\lambda_2(X_g)}{\sL_2(X_g)}$ uniformly comparable to $\frac{1}{\ln g}$. More precisely, let $\Prob$ be the probability measure on $\sM_g$ induced by the Weil-Petersson metric. We prove
\bt\label{mt-3}
The following limit holds:
\[\lim \limits_{g\to \infty}\Prob\left(X_g\in \sM_g; \  \frac{\lambda_2(X_g)}{ \sL_2(X_g)}\asymp \frac{1}{\ln g} \right)=1.\]
\et

It would be \emph{interesting} to know whether the following two limits exist: 
\[\lim \limits_{g\to \infty}\inf\limits_{X_g\in \sM_g}\frac{g^2\cdot \lambda_2(X_g)}{\sL_2(X_g)} \ \textit{and}  \ \lim \limits_{g\to \infty}\sup\limits_{X_g\in \sM_g}\frac{\lambda_2(X_g)}{\sL_2(X_g)}.\]

\subsection*{Strategy on the proof.} We briefly introduce the proof of Theorem \ref{mt-1} which is divided into two parts.

\underline{Optimal lower bound:}  we cut $X_g$ along a simple closed multi-geodesic that realizes $\mathcal{L}_1(X_g)$ into two components $M_1, M_2$. Then we consider their first eigenvalues for the Neumann boundary condition and aim to use the Max-minimal principle. \emph{When $\mathcal{L}_1(X_g)$ is uniformly bounded from above}, the lower bound in Theorem \ref{mt-1} is established and based on Theorem \ref{mt-2} which closely follows the methods in \cite{wu2021optimal}.  \emph{When $\mathcal{L}_1(X_g)$ is unbounded for large genus $g$}, we estimate  $\sigma_1(M_i)$ by the widths of maximal half-collars of boundary components. These two cases are separated as Proposition \ref{l2sigma2l} and \ref{l2sigma2l-2} in Section \ref{lowboundestimation}. To see that the lower bounds are optimal: actually the surface $\mathcal{X}_g$ constructed in \cite{WX18} satisfies that $\mathcal{L}_2(\mathcal{X}_g)\asymp 1$ and $\lambda_2(\mathcal{X}_g)\asymp \frac{1}{g^2}.$

\underline{Optimal upper bound:} as in \cite{SWY80} we take appropriate test functions $\psi$ with $\nabla \psi$ supporting on the collars of a simple closed multi-geodesic that realizes $\mathcal{L}_2(X_g)$ into the max-minimal principle. The upper bound in Theorem \ref{mt-1} follows from certain direct computations. The construction of $\mathcal{Y}_g$, realizing the upper bound in Theorem \ref{mt-1}, is from the perspective of probability.  By the work of Mirzakhani \cite{mirzakhani2013growth}, for large $g$ one may choose a hyperbolic surface $X_{g-2}\in\sM_{g-2}$ such that both its systole and Cheeger constant are uniformly positive, and moreover there exists one point in $X_{g-2}$ such that the injectivity radius at this point goes to infinity as $g\to \infty$. We remove that point, and consider the unique complete hyperbolic metric on that surface with the same complex structure as $X_{g-2}$. Then it is known by the work of Brooks \cite{brooks1999platonic} that this one-cusped hyperbolic surface of genus $(g-2)$ has uniformly positive Cheeger constant. Next we replace the cusp with a geodesic boundary and glue a fixed $S_{2,1}$ along it to get a hyperbolic surface $\mathcal{Y}_g \in \sM_g$. We shall show that the new surface $\mathcal{Y}_g$ satisfies that $\mathcal{L}_2(\mathcal{Y}_g)\asymp1$ and  $\lambda_2(\mathcal{Y}_g)\asymp 1$.

\subsection*{Notations.} For any two positive functions $f_1(g)$ and $f_2(g)$, we say $$f_1(g)\prec f_2(g) \quad \emph{or} \quad f_2(g)\succ f_1(g)$$ if there exists a universal constant $C>0$, independent of $g$, such that $f_1(g) \leq C \cdot f_2(g)$; and we say $$f_1(g) \asymp f_2(g)$$ if $f_1(g)\prec f_2(g)$ and $f_2(g)\prec f_1(g)$.

\subsection*{Plan of the paper.} Section \ref{prelina} will provide a review of relevant background materials, and present several useful properties on the geometry and spectra of hyperbolic surfaces, which will be applied in later sections. In Section \ref{inequalitiesfirst} we will prove Theorem \ref{mt-2}. The proof of Theorem \ref{mt-1} will split into two parts: we will firstly show the optimal lower bound of Theorem \ref{mt-1} in Section \ref{lowboundestimation}; and then show the optimal upper bound in Section \ref{upperbound}. In Section \ref{geometryquanityrandomsurface}, we will study the asymptotic behavior of $\frac{\lambda_2(X_g)}{\sL_2(X_g)}$ for random hyperbolic surfaces of large genus, and prove Theorem \ref{mt-3}. 

\subsection*{Acknowledgement}
We would like to thank Yuhao Xue for many helpful discussions, and especially for suggesting Example \ref{Examp}. The second named author is partially supported by the NSFC grant No. $12171263$.

\tableofcontents

\section{Preliminaries}\label{prelina}
In this section, we include some basis on 2-dimensional  hyperbolic geometry and  spectrum theory, along with preparative calculations used in this paper.

\subsection*{Bounds on $\sL_k(X_g)$} Recall that $\mathcal{M}_g$ is the moduli space of closed Riemann surface of genus $g$. Let $X_g\in \sM_g$ and $\mathcal{P}$ be a pants decomposition of $X_g$ consisting of $(3g-3)$ pairwise disjoint simple closed geodesics $\{\gamma_i\}_{i=1}^{3g-3}$. Set \[\ell_\infty(\mathcal{P})=\max_{1\leq i \leq 3g-3}\ell_{\gamma_i}(X_g).\]
The Bers' constant can be defined as 
$$
\mathcal{B}_g=\sup_{X_g\in\mathcal{M}_{g}} \inf_{\mathrm{pants \ decomp} \ \mathcal{P}} \ \ell_\infty(\mathcal{P}).
$$
In \cite{bers1985inequality}   it was shown that
$\mathcal{B}_g\leq 26(g-1)$ for $g\geq 2$. It easily follows that $\sL_k(X_g)\leq 78\cdot k\cdot g$ for $1\leq k\leq 2g-3$. However,  this estimate is not precise enough for our purpose.

In this subsection, we will prove the following two estimates.

\begin{proposition}
\label{xiamianyaoyong}
If $X_{g,n}$ is a hyperbolic surface of genus $g$ with $n$ geodesic boundary components, where $2g+n\geq 4$, and the length of each boundary component is less than or equal to $L$ for some constant $L>0$, then there exists a  constant $s(L)>0$ depending only on $L$ such that $$
\sL_1(X_{g,n})\leq s(L)\ln \left(\area(X_{g,n})\right).
$$
\end{proposition}

\begin{rem*}
As $L\to 0$, this in particular tells that Proposition \ref{xiamianyaoyong} also holds for cusped hyperbolic surfaces.  
\end{rem*}

\begin{proposition}\label{boundlk}
For any integer $k\in[1,2g-3]$, there exists a universal constant $c_k>0$ only depending on $k$ such that for any closed hyperbolic surface $X_g$ of genus $g$, we have 
 $$\sL_k(X_g)\leq c_k\ln \left(\area(X_{g})\right).$$
\end{proposition}

\begin{rem*}
The case of $k=1$ for Proposition  \ref{boundlk} has been proved in \cite{NWX20}, based on the result in \cite{sabourau2008asymptotic}. We will provide a new  and elementary proof for general $k$. Moreover, we will show that as $g\to \infty$, a generic hyperbolic surface $X_g \in \sM_g$ has $\sL_k(X_g)\succ \ln g$: see Proposition \ref{L2random} and its following remark.
\end{rem*}

Before proving them, we will introduce several  important results on hyperbolic geometry. The first one is the classical Collar Lemma.  see \eg \cite{keen1974collars}.

\begin{lemma}[Collar Lemma]\label{collar lemma} If $\gamma$ is a simple closed geodesic on  a  hyperbolic surface $X$, then  the neighborhood  $$
 T(\gamma)=\{x\in X:\dist(x,\gamma)\leq w(\gamma)\},
 $$
 called the collar of $\gamma$, is isometric to the cylinder $(\rho ,t)\in[-w(\gamma),w(\gamma)]\times \R/\Z$ with the metric $$
 ds^2=d\rho^2+\ell^2_\gamma(X)\cosh^2{\rho}\,dt^2,
 $$
 where $w(\gamma)=\arcsinh\left(\frac{1}{\sinh \left(\frac{1}{2}\ell_\gamma(X)\right)}\right)$ is the half width of the collar $T(\gamma)$.
 \end{lemma}
 
 Let $S_{g,n}$ be a fixed   surface of genus $g$ with $n$ boundary components.
 We use $\mathcal{M}_{g,n}(L)=\mathcal{M}(S_{g,n},\ell_{\alpha_i}=L_i)$ with $L=(L_1,\cdots,L_n)\in \mathbb{R}_{\geq 0}^n$ to represent the moduli space of  hyperbolic Riemann surfaces homeomorphic to $S_{g,n}$  with $n$ geodesic boundary components $(\alpha_1,\cdots,\alpha_n)$ of lengths $(L_1,\cdots,L_n)$.
 The following theorem, given by Parlier in \cite{parlier2005lengths},  will be frequently  applied in this paper to deal with hyperbolic surfaces with boundaries. 
 
 \begin{theorem}[Parlier]\label{parlierlength}
Use $L\leq \tilde{L}$ to represent  $L_i\leq \tilde{L}_i$ for any $1\leq i\leq n$,
 where $L=(L_1,\cdots,L_n)$ and  $\tilde{L}=(\tilde{L}_1,\cdots,\tilde{L}_n)$.
 For any $X_{g,n}\in\mathcal{M}_{g,n}(L)$,  if $L\leq \tilde{L}$, then there is an $\tilde{X}_{g,n}\in \mathcal{M}_{g,n}(\tilde{L})$ such that for any simple closed curve $\gamma$ on $S_{g,n}$, $$\ell_\gamma(X_{g,n})\leq \ell_{\gamma}(\tilde{X}_{g,n}).$$
 Similarly, for any $X_{g,n}\in\mathcal{M}_{g,n}(L)$,  if $L\geq \tilde{L}$, then there is an $\tilde{X}_{g,n}\in \mathcal{M}_{g,n}(\tilde{L})$ such that for any simple closed curve $\gamma$ on $S_{g,n}$, $$\ell_\gamma(X_{g,n})\geq \ell_{\gamma}(\tilde{X}_{g,n}).$$
Here the length of a boundary component is allowed to be zero, representing a puncture instead of a boundary geodesic.
\end{theorem}
 For a  hyperbolic surface $X$,  its \emph{systole} $\sys(X)$ is  the infimum of all lengths of simple closed geodesics that are not homotopic to any boundary component of $X$.  An easy area argument shows that for any closed hyperbolic surface $X_g$,
 \begin{equation}\label{systolebounded}
     \sys (X_g)\leq  2\ln (4g-2)
     \leq 6\ln g.
 \end{equation}
 
Proposition \ref{xiamianyaoyong} will be used  to  prove Proposition \ref{boundlk}, and will also be used in the later part of this article.
\bp[Proof of Proposition \ref{xiamianyaoyong}]
In fact, we will prove that
\begin{equation}\label{infactineq}
\sL_1(X_{g,n})\leq 2L+4\ln\left(1+\frac{4\area(X_{g,n})}{L}\right).
\end{equation}

Fix a boundary geodesic $\gamma_0\subset \partial X_{g,n}$ and consider its length. 

\underline{Case-1:} \textit{$\ell_{\gamma_0}(X_{g,n})\geq \frac{1}{2}L$.} Consider the width of the maximal half collar of $\gamma_0$. That is, the maximal $w$ such that $N_{\gamma_0}(w)=\{x\in X_{g,n}|\dist(x,\gamma_0)<w\}$ is isometric to $[0,w)\times S^1$ with  the metric $d\rho^2+\cosh^2\rho\, \ell^2_{\gamma_0}(X_{g,n})dt^2$ as in  Lemma \ref{collar lemma}. The area of the half collar  is $\ell_{\gamma_0}(X_{g,n})\sinh w\leq \area(X_{g,n})$, so we have $\frac{e^w-1}{2}\leq \frac{2\area(X_{g,n})}{L}$, which implies 
\begin{equation}\label{w0L}
    w\leq w_0:=\ln\left(1+\frac{4\area(X_{g,n})}{L}\right).
\end{equation}
\begin{figure}[h]
    \centering
    \includegraphics[width=2.5 in]{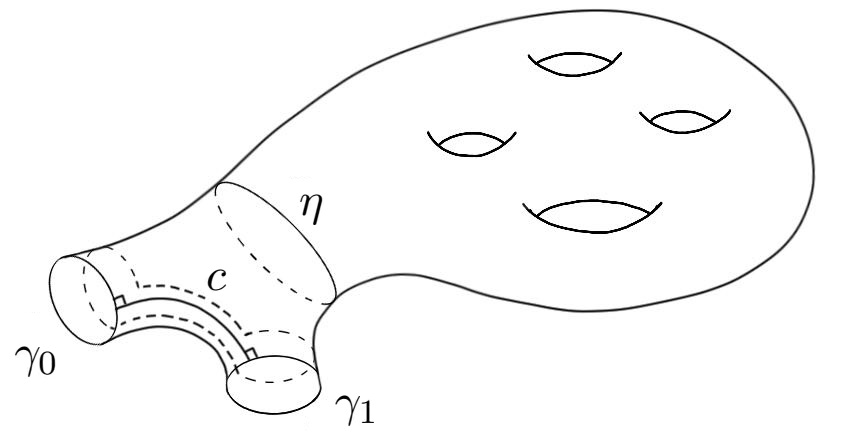}
    \caption{Subcase-$1$}
    \label{fig:lemma33case1}
\end{figure}Then two sub-cases may occur. Firstly, if the boundary of $N_{\gamma_0}(w)$ touches another boundary component $\gamma_1$, then there is  a shortest geodesic $c$ connecting $\gamma_0$ and $\gamma_1$ of length $w$. 
Take $\delta =\gamma_0\cup c\cup \gamma_1$. The boundary of its $\epsilon$ neighborhood $\partial N_\delta(\epsilon)$ for small $\epsilon$ will be  homotopic to a simple closed geodesic $\eta$. See Figure \ref{fig:lemma33case1}. Since, along with $\{\gamma_0,\gamma_1\}$,  the curve $\eta$  bounds a pair of pants and  the length of $\eta$ is bounded from above by $\ell_{\gamma_0}(X_{g,n})+\ell_{\gamma_1}(X_{g,n})+2w$, we have  $$\sL_1(X_{g,n})\leq \ell_\eta(X_{g,n})\leq 2L+2w_0(L).$$
Otherwise,  if the boundary of $N_{\gamma_0}(w)$ doesn't touch another boundary component, then there is a geodesic arc $c$ connecting two different  points $\{p,q\}$ on $\gamma_0$ of length $2w$. The two points  $\{p,q\}$  separate $\gamma_0$ into two parts $\gamma_1$ and $\gamma_2$. Then  $\gamma_1\cup c$ and $\gamma_2\cup c$ will be  homotopic to   simple closed geodesics $\delta_1$ and $\delta_2$ respectively of lengths  both less than $\ell_{\gamma_0}(X_{g,n})+2w$. See Figure \ref{fig:lemma33case2}.  Since   $\delta_1,\delta_2$ along  with $\gamma_0$  bound a pair of pants, we have $$\sL_1(X_{g,n})\leq \ell_{\delta_1}(X_{g,n})+\ell_{\delta_2}(X_{g,n})\leq 2L+4w_0(L).$$ 

\begin{figure}[h]
    \centering
    \includegraphics[width=2.5 in]{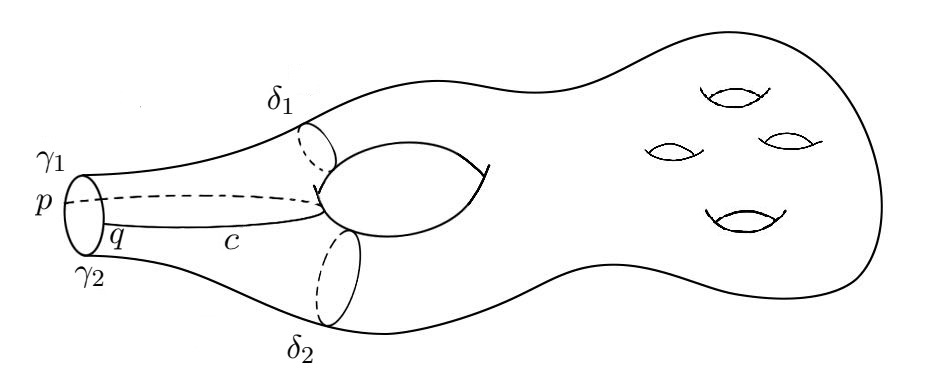}
    \caption{Subcase-$2$}
    \label{fig:lemma33case2}
\end{figure}

\underline{Case-2:} \textit{$\ell_{\gamma_0}(X_{g,n})\leq \frac{1}{2}L$.}  By Theorem \ref{parlierlength}, there is a hyperbolic surface $\Tilde{X}_{g,n}$ homeomorphic to $X_{g,n}$, satisfying that the boundary geodesic corresponding to $\gamma_0$ on $\tilde{X}_{g,n}$ is of length $\frac{L}{2}$, and all interior simple closed geodesics will not be shorter. By the same argument as in Case-1, we have $\sL_1(\tilde{X}_{g,n})\leq 2L+4w_0(L)$. Since the union of simple closed curves corresponding to   a multi-curve realizing $\sL_1(\tilde{X}_{g,n})$ will also separate $X_{g,n}$, we have  $$\sL_1(X_{g,n})\leq \sL_1(\tilde{X}_{g,n})\leq 2L+4w_0(L).$$ 

Substitute \eqref{w0L} into the inequality $\sL_1(X_{g,n})\leq 2L+4w_0(L)$ above, we obtain \eqref{infactineq}. Then the conclusion follows by taking a suitable $s(L)>0$ only depending on $L$.
\ep
Now we are ready to prove Proposition \ref{boundlk}.
\begin{proof}[Proof of Proposition \ref{boundlk}]
 For $k=1$, consider a simple closed geodesic $\gamma$ with $\ell_{\gamma}(X_g)=\sys(X_g)$. If $\gamma$ is separating, that is, $X_g-\gamma$ is not connected, then by (\ref{systolebounded}) we have \be\label{prop31}
\sL_1(X_g)\leq \ell_{\gamma}(X_g)\leq 6\ln g.\ene
Now assume that $\gamma$ is not separating. We firstly assume that $\ell_\gamma(X_g)< 1$, then $X_{g}-\gamma$ is of type $(g-1,2)$ with each boundary component shorter than $1$. So by Proposition \ref{xiamianyaoyong}, we have $\sL_1(X_g-\gamma)\leq s(1)\ln (4\pi(g-1))$.  The  multi-curve $\eta$ realizing $\sL_1(X_{g}-\gamma)$, along with $\gamma$, will separate $X_g$ into two parts. Therefore, \be\label{prop32}
\sL_1(X_g)\leq s(1)\ln (4\pi(g-1))+1.
\ene
If $\gamma$ is non-separating and $1\leq \ell_\gamma(X_g)<6\ln(g)$, then consider the half width $w$ of its maximal collar $(-w,w)\times S^1$  as described in Proposition \ref{xiamianyaoyong}.
The area of the maximal collar is $2\ell_\gamma(X_g)\sinh w\leq 4\pi(g-1).$ Then $e^w-1\leq 2\sinh w\leq 4\pi(g-1)$, implying that 
\be\label{prop33}
 w\leq 5\ln g.\ene  
 By the definition of the maximal collar, there is  a geodesic arc $c$ of length $2w$ with its endpoints on $\gamma$. Consider the $\epsilon$ neighborhood  $N_\epsilon(c\cup \gamma)$ of $c\cup \gamma$ for small $\epsilon$. The boundary $\partial N_\epsilon(c\cup \gamma)$ will be homotopic to some simple closed geodesics that bound a pair of pants along with $\gamma$, with the total length bounded from above by $4w+2\ell_
 \gamma(X_g).$ It follows from (\ref{prop31}) and (\ref{prop33}) that \be\label{prop34}
 \sL_1(X_g)\leq 4w+3\ell_\gamma(X_g)\leq 38\ln g. 
 \ene
Recall that $\area(X_g)=4\pi(g-1)$ by Gauss-Bonnet. Therefore it follows from \eqref{prop32} and \eqref{prop34} that  
\begin{equation*}
\sL_1(X_g)\leq c_1\ln \left(\area(X_g)\right)
\end{equation*}
for a suitable universal constant $c_1>0$.

Now we show the way to find $c_k$  based on the estimate for $k-1$ if $k\geq 2$. Let $\gamma=\cup_{j=1}^t\gamma_j$ be the multi-curve realizing $\sL_{k-1}(X_g)$ which separates $X_g$ into $k$ components  $\{M_i\}_{i=1}^k$. By  assumption, we have
\be\label{prop35}
\ell_{\partial M_i}(M_i)\leq \ell_\gamma(X_g)=\sL_{k-1}(X_g)\leq c_{k-1}\ln \left(\area(X_g)\right).
\ene
For large $g$, at least one component $M_i$ is not of type $S_{1,1}$ or $S_{0,3}$, and we can assume it is $M_1$. Then it
 follows from (\ref{infactineq})  and (\ref{prop35}) that \be\label{prop36}
\sL_1(M_1)\leq 2c_{k-1}\ln \left(\area(X_g)\right) +4\ln\left( 1+\frac{4\area(M_1)}{c_{k-1}\ln \left(\area(X_g)\right)}  \right). 
\ene
Along with $\gamma$, the multi-curve realizing $\sL_1(M_1)$ will separate $X_g$ into $k+1$ components. So by (\ref{prop35}) and (\ref{prop36}) we have
\begin{equation*}
\sL_k(X_g)\leq 3c_{k-1}\ln \left(\area(X_g)\right)+4\ln\left( 1+\frac{4\area(M_1)}{c_{k-1}\ln \left(\area(X_g)\right)}  \right),
\end{equation*}
which implies the conclusion  \begin{equation*}
\sL_{k}(X_g)\leq c_k\ln\left(\area(X_g)\right)
\end{equation*}
for suitable choice of $c_k$ only depending on $k$, since $\area(M_1)\leq \area(X_g)=4\pi(g-1)$.
\end{proof}

\subsection*{Eigenvalues for geometric Laplacian} 
  Use $\phi_k$ to represent the $k^{th}$ normalized eigenfunction of $\Delta$ with respect to    the $k^{th}$ eigenvalue $\lambda_k(X_g)$ of $X_g$. It is easy to see that  $\phi_0=\frac{1}{\sqrt{\area(X_g)}}$. For $k\geq 1$, the $k^{th}$ eigenvalue of $\Delta$ can  be expressed as  
 $$
 \lambda_k(X_g)=\inf\limits_{\left<f,\phi_i\right>=0,i<k} \frac{\int_{X_g}|\nabla f|^2d\mu}{\int_{X_g}f^2d\mu},
 $$
 where $\left<f,\phi_i\right>$ means $\int_{X_g}f\phi_id\mu$, and $f$ runs over the Sobolev space  $W^{1,2}(X_g)$ or $C^\infty(X_g)$, which is dense in $W^{1,2}(X_g)$. The infimum of this Rayleigh quotient  can be achieved by some   eigenfunction $\phi_k\in C^\infty(X_g)$.
 
For  $X_{g,n}$ with non-empty boundary $\partial X_{g,n}$, the first positive eigenvalue $\sigma_1(X_{g,n})$ for Neumann boundary condition is  $$
 \sigma_1(X_{g,n})=\inf\limits_{\int_{X_{g,n}} fd\mu=0} \frac{\int_{X_{g,n}}|\nabla f|^2d\mu}{\int_{X_{g,n}}f^2d\mu},
 $$
 where $f$ runs over $W^{1,2}(X_{g,n})$. The infimum can be achieved by some Neumann eigenfunction $f\in C^\infty(X_{g,n})$, and the partial derivative $\frac{\partial f }{\partial \Vec{n}}$ in the normal direction  will vanish  at $\partial X_{g,n}$.
 
  Generally the  $k^{th}$ eigenvalue can be estimated by the following max-minimal  principle. See \eg  \cite[Theorem $3.34$]{bergeron2016spectrum} for its proof.

\begin{theorem}[Max-minimal  principle]\label{maxminpr}
(1) If $X=N_1\cup\cdots\cup N_k $ is a partition of surface $X$, with $\area(N_i)>0$ and $\area(N_i\cap N_j)=0$, then $$
\lambda_{k}(X)\geq \min_{1\leq i\leq k}\sigma_1(N_i).
$$
(2) If $f_1,\cdots,f_{k+1}$ are $k+1$ functions on $X$ with $L^2$ norm $1$ and the supports of any two of them have intersections of measure zero, then $$
\lambda_k(X)\leq \max\limits_{1\leq i\leq k+1}\int_{X}|\nabla f_i|^2d\mu.$$
\end{theorem}

The Cheeger inequality \cite{cheeger1970lower} gives lower bounds for first eigenvalues.

\begin{definition*}
For a surface  $\Omega $ with possibly non-empty boundary, the \textit{Cheeger constant} $h(\Omega)$ of $\Omega$ is defined as 
$$
h(\Omega)=\inf_{\Gamma} \frac{\ell(\Gamma \cap \mathring{\Omega})}{\min\{ \area(A),\area(B)\}},
$$
where $\mathring{\Omega}$ is the interior of $\Omega$ ($\mathring{\Omega}=\Omega$ if $\Omega$ is closed), and $\Gamma$  is any   set of  piecewise smooth curves separating $\Omega$ into two parts $A$ and $B$. It is known by an observation due to Yau that $A$ and $B$ can be chosen to be connected.
\end{definition*}

\begin{theorem}[Cheeger inequality]\label{cheineq}
If  $\Omega=X_{g,n}$ is a hyperbolic surface with non-empty boundary, then
$$\sigma_1(\Omega)\geq\frac{1}{4}h^2(\Omega).$$
If  $\Omega=X_g$ is a closed hyperbolic surface,  then
$$\lambda_1(\Omega)\geq\frac{1}{4}h^2(\Omega).$$

 \end{theorem}

\subsubsection{Eigenvalues for $S_{1,1}$ and $S_{0,3}$} For a compact hyperbolic surface $X$ of type $S_{1,1}$ or $S_{0,3}$, the geometric quantity $\sL_1(X)$ is not meaningful. Its first eigenvalue with Neumann boundary condition can be controlled solely by its boundary lengths.  Before proving that, we recall the following two classical  results.
The first one is the  isoperimetric inequality.  See \eg Section $8.1$ in \cite{buser2010geometry}.

\begin{lemma}[Isoperimetric inequalities]\label{isoperimetric}
 Let $\Omega$ be  a domain with piecewise smooth boundary $\partial \Omega$, equipped with the hyperbolic metric. If $\Omega$ is topologically a disk or a cylinder,  then 
$$
\ell(\partial \Omega)\geq \area(\Omega).
$$
\end{lemma}

The second one is as follows. 
See \eg Theorem 4.2.2 in \cite{buser2010geometry} for details.
\begin{lemma}\label{nonsimple4arc}
If $\gamma\subset X$ is a non-simple closed geodesic on  $X$, then $$\ell_\gamma(X)\geq 4\arcsinh1.$$
\end{lemma}

With the help of the  two lemmas above, we will prove two   lower bounds of the first   eigenvalues with respect to  Neumann boundary conditions for $S_{1,1}$ and $S_{0,3}$, which will be applied later.

\bl\label{eigens03s11}
If $X$ is of type $S_{0,3}$ or $S_{1,1}$ with each boundary component of  length less than  $L$, then there exists a constant $K(L)>0$ only depending on $L$ such that $$\sigma_1(X)\geq K(L).$$
\el
\bp
 Consider a division of $X$ into $A\cup B$ with  $\partial A \cap \mathring{X}=\partial B\cap \mathring{X}=\mathring{\Gamma}$, where $\mathring{\Gamma}=\Gamma \setminus \partial \Gamma$ is the interior of the (multi)-curve $\Gamma$ obtained by removing its boundary points on $\partial X$, and $h(X)\leq \frac{\ell(\Gamma)}{\min\{\area(A),\area(B)\}}\leq 2h(X)$. It follows that $$
 h(X)\geq \frac{\ell(\Gamma)}{2\pi}.
 $$
 By the Collar Lemma, \ie, Lemma \ref{collar lemma}, each boundary geodesic admits a half collar of width $w(L)>0$.
 If $\Gamma$  intersects with the boundary $\partial X$ and is not homotopic to an arc on $\partial X$, then $\ell(\Gamma)\geq 2w(L)$. Therefore, we have \be\label{s0311}
 h(X)>\frac{w(L)}{\pi}.\ene
If $\Gamma$ contains a non-simple component  and  $\Gamma\cap \partial   X=\emptyset $, then by Lemma \ref{nonsimple4arc} $\ell(\Gamma)\geq 4\arcsinh 1.$ Therefore, we have
\be \label{s0312}
h(X)>\frac{2\arcsinh 1}{\pi}.
\ene
If neither of the the above two situations occurs, then $\Gamma$ will bound  disks or  cylinders, assuming the union of which  is $A$.  By  Lemma \ref{isoperimetric}, we have $\area(A)\leq \ell(\partial A)$. However, some parts of  $\partial A$ may lie  on  $\partial X$: those arcs are homotopic to some sub-arcs of $\Gamma$. Therefore we always have $\ell(\partial A)\leq 2 \ell(\Gamma)$.
This gives \be\label{s0313}
h(X)\geq \frac{1}{4}.\ene
Then the  lemma follows from  the  Cheeger inequality, \ie,  Theorem  \ref{cheineq}, along with equations  (\ref{s0311}), (\ref{s0312}) and (\ref{s0313}).
\ep
Viewing $X$ as a subsurface of $X_g$,
if the boundary lengths of $X$ are not bounded from above by any  constant $L$, but  grow slowly as $g\to \infty$, then we have the following estimate.

\begin{lemma}\label{s03s11lngsigma}
If  $X$ is of  type $S_{1,1}$ or $S_{0,3}$, with each boundary component of  length less than   $\ln g$ for $g\geq 2$, then we have
$$\sigma_1(X)\succ \frac{1}{g}.
$$
\end{lemma}
\begin{proof}
By  Lemma \ref{collar lemma}, each boundary component of  $X$ admits a half collar of width $$
\begin{aligned}
w=&\arcsinh\frac{1}{\sinh\frac{\ln g}{2}}
\geq \arcsinh\frac{2}{g^{\frac{1}{2}}}\\
=&\ln\left(\frac{2}{g^{\frac{1}{2}}}+\sqrt{1+\frac{4}{g}   }\right)
\geq \frac{k}{g^\frac{1}{2}}\\
\end{aligned}
$$
for some uniform constant $k>0$. Consider a division of $X$ into $A\cup B$ with  $\partial A \cap \mathring{X}=\partial B\cap \mathring{X}=\mathring{\Gamma}$ and $h(X)\leq \frac{\ell(\Gamma)}{\min\{\area(A),\area(B)\}}\leq 2h(X)$.
If some part  $\gamma$ in $\Gamma$ intersects $\partial X$ and isn't homotopic to an arc on $\partial X$ , then $\ell(\gamma)\geq 2w$, so $
h(X)\geq \frac{w}{\pi}.
$
If $\Gamma$ have non-simple closed part, by Lemma \ref{nonsimple4arc} we have $  \ell(\Gamma)\geq 4\arcsinh 1$, so $h(X)\geq \frac{2\arcsinh 1}{\pi}$.
Otherwise, similar to the proof of Lemma \ref{eigens03s11}, it follows from Lemma \ref{isoperimetric} that $\frac{\ell(\Gamma)}{\min\{\area(A),\area(B)\}}\geq \frac{1}{2}$, which means $h(X)\geq \frac{1}{4}$. Then it follows from Theorem \ref{cheineq} that 
$$
\sigma_1(X)\geq \min\left\{\frac{1}{64}, \ \left(\frac{\arcsinh 1}{\pi}\right)^2, \ \frac{k^2}{4\pi^2g}\right\}\asymp \frac{1}{g},
$$
which completes the proof.
\end{proof}

\section{Uniform lower bounds for first positive Neumann eigenvalues}\label{inequalitiesfirst}

It was proven in \cite{wu2021optimal} that  
\begin{theorem}[Wu-Xue]\label{mt-WX-1}
For any closed hyperbolic surface $X_g$ of genus $g$, we have
\[\lambda_1(X_g)\succ \frac{\sL_1(X_g)}{g^2}.\]
\end{theorem}

In this section we prove Theorem \ref{mt-2}, extending the above result to first Neumann eigenvalues for compact hyperbolic surface $X_{g,n}$ with non-empty geodesic boundaries, where $X_{g,n}$ is not of type $S_{1,1}$ and $S_{0,3}$. This plays a key role in the proof of our main result Theorem \ref{mt-1}. We follow the idea in \cite{wu2021optimal} to prove Theorem \ref{mt-2}. Before proving it, we provide the following two propositions as summaries of special cases encountered in the proof. Firstly we  fix a uniform constant 
\begin{equation}
0<\epsilon<0.05.
\end{equation}

\begin{proposition}\label{keypartepsilon}
Let $X_{g,n}$ be the hyperbolic surface in Theorem \ref{mt-2} and satisfy $\sL_1(X_{g,n})\geq \epsilon$. Then there is a constant $\hat{c}_\epsilon(L)>0$ only depending on $L$ and $\epsilon$ such that $$
\frac{\sigma_1(X_{g,n})}{\sL_1(X_{g,n})}\geq \frac{\hat{c}_\epsilon(L)}{|\area(X_{g,n})|^2}.
$$\end{proposition}
\bp
 Consider a division of $X_{g,n}$ into $A\cup B$ with   $\partial A \cap \mathring{X}_{g,n}=\partial B\cap \mathring{X}_{g,n}=\mathring{\Gamma}$ such that $\area(A)\leq \area(B)$ and $h(X_{g,n})\leq \frac{\ell(\Gamma)}{\area(A)}\leq 2h(X_{g,n})$.
 
\underline{Case-1:} \textit{$A$ contains only disks or cylinders.} Similar as in the proof of Lemma \ref{eigens03s11}, it follows from Lemma  \ref{isoperimetric} that $\frac{\ell(\Gamma)}{\area(A)}\geq \frac{1}{2}$. Hence, $h(X_{g,n})\geq \frac{1}{4}$. By  Proposition \ref{xiamianyaoyong} and  Theorem \ref{cheineq}, we have
\begin{equation}\label{122}
\frac{\sigma_1(M)}{\sL_1(X_{g,n})}
    \geq \frac{1}{64s(L)\ln(|\area(X_{g,n})|)}
    \geq \frac{1}{64s(L)}\frac{1}{|\area(X_{g,n})|^2}.
\end{equation}

\underline{Case-2:} \textit{$\Gamma\cap \partial X_{g,n}=\emptyset$ and $A$ contains some component which is not topologically  a disk or a cylinder.} In this case $\Gamma$ contains a homotopically non-trivial closed curve which is not a boundary component. In particular, $\ell(\Gamma)\geq \sL_1(X_{g,n})$.  By Theorem \ref{cheineq} and the assumption $\sL_1(X_{g,n})\geq \epsilon$, we have 
\begin{equation}\label{125}
    \sigma_1(X_{g,n})\geq \frac{1}{16}\frac{\ell^2(\Gamma)}{|\area(A)|^2}\geq \frac{\epsilon}{4|\area(X_{g,n})|^2}\cdot \sL_1(X_{g,n}).
\end{equation}

\underline{Case-3:} \textit{$A$ is not of the types in case-1 or case-2.} In this case,
 $A$ contains a connected component $A_0$ such that $\partial A_0$ contains
a simple  curve $\Gamma_0$ which intersects with $\partial X_{g,n}$, and does not bound a disk along with $\partial X_{g,n}$. Since  for any component $\gamma\subset \partial X_{g,n}$, $\ell_\gamma(X_{g,n})\leq L$, by Lemma \ref{collar lemma}, all boundary components admit a half collar of width $w(L)=\arcsinh\left(\frac{1}{\sinh \frac{L}{2}}   \right)$. Hence we have $$\ell(\Gamma_0)\geq 2w(L).$$
If $\Gamma_0$ connects two different boundary components $\alpha$ and $\beta$ at $p_0\in \alpha$ and $q_0\in \beta$, then we consider  $\partial N_\epsilon(\alpha\cup\Gamma_0\cup\beta)$, which is homotopic to a simple closed geodesic $\delta$ of length $$\ell_\delta(X_{g,n})\leq \ell_\alpha(X_{g,n})+\ell_\beta(X_{g,n})+2\ell(\Gamma_0) \leq  2L+2\ell(\Gamma_0).$$ Along with $\{\alpha,\beta\}$, $\delta$ bounds a pair of pants. Since $2g+n\geq 4$, $\delta$ is not a boundary geodesic; otherwise $X_{g,n}\cong S_{0,3}$. So we have $$\sL_1(X_{g,n})\leq \ell_\delta(X_{g,n}).$$ In this situation, by Theorem \ref{cheineq} and $\ell(\Gamma_0)\geq 2w(L)$, we have
\begin{equation}\label{123}
\begin{aligned}
    \frac{\sigma_1(X_{g,n})}{\sL_1(X_{g,n})}& \geq \frac{1}{4(2L+2\ell(\Gamma_0))}\left(\frac{\ell(\Gamma_0)}{2\area(A)}\right)^2\\
    &\geq   \frac{w^2(L)}{(2L+4w(L))|\area(X_{g,n})|^2}.\\
    \end{aligned}
\end{equation}
If $\Gamma_0$ has two endpoints on the same boundary component $\alpha$, and the two endpoints cut  $\alpha$  into $\alpha_1$ and $\alpha_2$, then
the piecewise smooth curves $\Gamma_0\cup \alpha_1$ and $\Gamma_0\cup \alpha_2$ will be homotopic to simple closed geodesics $\delta_1$ and $\delta_2$ respectively of lengths $$\ell_{\delta_i}(X_{g,n})\leq \ell(\Gamma_0)+L$$ for $i=1,2$. Along with $\alpha$, $\{\delta_1,\delta_2\}$
 will bound a pair of pants. Since $2g+n\geq 4$, $\delta_1$ and $\delta_2$ can not simultaneously be boundary geodesics; otherwise $X_{g,n}\cong S_{0,3}$. Since $X_{g,n}$ is not of type $S_{1,1}$, $\{\delta_1,\delta_2\}$ separates $X_{g,n}$. In particular, we have
$$\sL_1(X_{g,n})\leq \ell_{\delta_1}(X_{g,n})+\ell_{\delta_2}(X_{g,n}).$$  In this situation, by Theorem \ref{cheineq} and $\ell(\Gamma_0)\geq 2w(L)$, we also have
\begin{equation}\label{124}
   \begin{aligned}
    \frac{\sigma_1(X_{g,n})}{\sL_1(X_{g,n})}
    &\geq \frac{1}{4(2L+2\ell(\Gamma_0))}\left(\frac{\ell(\Gamma_0)}{2\area(A)}\right)^2\\
    &\geq   \frac{w^2(L)}{(2L+4w(L))|\area(X_{g,n})|^2}.\\
    \end{aligned}
\end{equation}

In conclusion, it follows from (\ref{122}), (\ref{125}), (\ref{123}) and (\ref{124}) that
\begin{equation*}
    \frac{\sigma_1(X_{g,n})}{\sL_1(X_{g,n})}\geq \frac{1}{|\area(X_{g,n})|^2}\cdot \min \left\{\frac{1}{64s(L)}, \ \frac{\epsilon}{4},\ \frac{w^2(L)}{2L+4w(L)}\right\}.
\end{equation*}

The proof is completed by setting $\hat{c}_\epsilon(L)=\min \left\{\frac{1}{64s(L)}, \ \frac{\epsilon}{4}, \ \frac{w^2(L)}{2L+4w(L)}\right\}$.
\ep

\begin{rem*}
The assumption that $\epsilon<0.05$  is not required in Proposition \ref{keypartepsilon}. Any fixed constant $\epsilon>0$ is enough. 
\end{rem*}

To prove Theorem \ref{mt-2}, in light of Proposition \ref{keypartepsilon} it suffices to prove the case where $\sL_1(X_{g,n})$ is arbitrarily small, which is implicitly contained in \cite{wu2021optimal}. More precisely, by  Sobolev embedding  theorem\cite{taylor2011partial}  there  exists a constant $c(\epsilon)>0$ only depending on $\epsilon$ such that $$
||\nabla \phi||_{L^\infty(B(\frac{\epsilon}{2}))}\leq c(\epsilon)\sum_{j=0}^\infty||\Delta^j(d\phi)||_{L^2(B(\epsilon))}
$$
for any smooth function $\phi$ on $B(\epsilon)$, a hyperbolic ball in $\mathbb{H}$ of radius $\epsilon$.
\begin{proposition}\label{keypartwuxue}
Let $X_{g,n}$ be the hyperbolic surface in Theorem \ref{mt-2} and satisfy that $\sL_1(X_{g,n})\leq \epsilon$, $\sigma_1(X_{g,n})\leq \frac{1}{4}$, $ \sigma_1(X_{g,n})\leq \frac{1}{1000\epsilon}\frac{\sL_1(X_{g,n})}{|\area(X_{g,n})|^2}$ and $c_1(\epsilon)\sqrt{\frac{\sL_1(X_{g,n})}{1000\epsilon}}\leq \frac{1}{64}$ where $c_1(\epsilon)=\frac{14\epsilon c(\epsilon)}{\sqrt{\area(B(\epsilon))}}$.
 Then we have
$$\frac{\sigma_1(X_{g,n})}{\sL_1(X_{g,n})}\geq \frac{1}{2^{14}|\area(X_{g,n})|^2}.$$
\end{proposition}

\bp  We only outline the proof here. One may see \cite{wu2021optimal} for more details. Denote  $\gamma_1,\cdots,\gamma_m$ to be all simple closed geodesics in the interior of  $X_{g,n}$ that are shorter than $2\epsilon$. Set $$B=\cup_{i=1}^mN(\gamma_i,w(\gamma_i)-2)$$ which is the disjoint union of collars centered at $\gamma_i$ of width $w(\gamma_i)-2$. Set $$A=\overline{  X_{g,n}-B}.$$ Denote all components of $A$ by $M_1,\cdots,M_m$ for some $m>0$. Let $\phi$ be a normalized  eigenfunction corresponding to $\sigma_1(X_{g,n})\leq \frac{1}{4}$ and
set $$
\mathrm{osc}(i)=\max_{p\in M_i}\phi(p)-\min_{p\in M_i}\phi(p)
$$
to be the oscillation of $\phi$ on $M_i$. Then the proof of \cite[Proposition 12]{wu2021optimal} gives that
\be\label{key1}
\sum_{i=1}^n \mathrm{osc}(i)\leq c_1(\epsilon)\sqrt{\area(X_{g,n})\sigma_1(X_{g,n})}\leq \frac{1}{64\sqrt{|\area(X_{g,n})|}}.
\ene
Since $\sigma_1(X_{g,n})\leq \frac{1}{1000|\area(X_{g,n})|^2}\leq \frac{1}{1000|\area(X_{g,n})|},$ the proof of \cite[Proposition 14]{wu2021optimal} yields that
$$
\sup_{p\in A}|\phi(p)|\geq \frac{1}{32\sqrt{|\area(X_{g,n})|}}.
$$
Along with \cite[Lemma 13]{wu2021optimal}, we have  \be\label{key2}
\sup_{p\in A}\phi(p)-\inf_{p\in A}\phi(p)\geq \frac{1}{32\sqrt{|\area(X_{g,n})|}}.
\ene
Combining (\ref{key1}) and (\ref{key2}), we have \be\label{key3}
\max_{p\in A}\phi(p)-\min_{p\in A}\phi(p)-\sum_{i=1}^n\mathrm{osc}(i)\geq \frac{1}{64\sqrt{|\area(X_{g,n})|}}.
\ene
As in \cite{wu2021optimal}, for each $i$ we assume that the image $\phi(M_i)=[a_i,b_i]$. Then $b_i-a_i=\mathrm{osc}(i)$. Denote $T_{ij}^1,\cdots,T_{ij}^{\theta_{ij}}$ to be all collar components of $B$ with center geodesics $\gamma_{ij}^\theta$ shorter than $2\epsilon$ and two boundaries $\Gamma_{ij}^{\theta,1},\Gamma_{ij}^{\theta,2}$  lying on $M_i$ and $M_j$. Take $$\delta_{ij}=\dist([a_i,b_i],[a_j,b_j]),$$ then it follows by \cite[Lemma 8]{wu2021optimal} and the Cauchy-Schwarz inequality that \be\label{key4}
\begin{aligned}
\sigma_1(X_{g,n})&=\int_{X_{g,n}}|\nabla \phi|^2 \geq \sum_{T_{ij}^\theta}\int_{T_{ij}^\theta}|\nabla \phi|^2\geq \sum_{T_{ij}^\theta}\frac{\ell_{\gamma_{ij}^\theta}(X_{g,n})}{4}\delta_{ij}^2\\
& \geq \frac{1}{4(3g-3+n)}\left(\sum_{T_{ij}^\theta} \sqrt{\ell_{\gamma_{ij}^\theta}(X_{g,n})}\delta_{ij}\right)^2.\\
\end{aligned}\ene
Rewrite $$\cup_{i=1}^m [a_i,b_i]=\cup_{k=1}^N [e_k,f_k]$$ with pairwise disjoint intervals $\{[e_k,f_k]\}_{k=1}^N$. For each $i\in [1,m]$ there exists $k_i\in[1,N]$ such that $$
[a_i,b_i]\subset [e_{k_i},f_{k_i}].
$$
Then the proof of \cite[Proposition 17]{wu2021optimal} gives that there is an integer $K\in [1,m-1]$ such that \be\label{key5}
\begin{aligned}
&\sum_{T_{ij}^\theta} \sqrt{\ell_{\gamma_{ij}^\theta}(X_{g,n})}\delta_{ij}\\
\geq& \left(\sum_{1\leq k_i\leq K<k_j\leq N} \sqrt{\ell_{\gamma_{ij}^\theta}(X_{g,n})}\right)\left(\max_{p\in A}\phi(p)-\min_{p\in A}\phi(p)-\sum_{i=1}^nosc(i)\right).
\end{aligned}
\ene
Same as \cite{wu2021optimal}, the union of the central closed geodesics of all $T_{ij}^\theta$ with $k_i\leq K<k_j $ will cut $X_{g,n}$ into at least two components. Therefore we have 
\be\label{key6}
\sum_{1\leq k_i\leq K<k_j\leq N} \sqrt{\ell_{\gamma_{ij}^\theta}(X_{g,n})}\geq \sqrt{\sum_{1\leq k_i\leq K<k_j\leq N} \ell_{\gamma_{ij}^\theta}(X_{g,n})}\geq \sqrt{\sL_1(X_{g,n})}.
\ene
Then the conclusion follows from (\ref{key3}), (\ref{key4}), (\ref{key5}) and (\ref{key6}). This completes the proof.
\ep

Now we are ready to prove Theorem \ref{mt-2}.
\bp[Proof of Theorem \ref{mt-2}]
If  $\sigma_1(X_{g,n})\geq \frac{1}{4}$,  by Proposition \ref{xiamianyaoyong}, we have \begin{equation*}
    \frac{\sigma_1(X_{g,n})}{\sL_1(X_{g,n})}\geq \frac{1}{4s(L)\ln |\area(X_{g,n})|}\geq \frac{1}{4s(L)|\area(X_{g,n})|^2}
    .
\end{equation*}

If
 $\sigma_1(X_{g,n})\leq \min\left\{ \frac{1}{4}, \frac{1}{1000\epsilon}\frac{\sL_1(X_{g,n})}{|\area(X_{g,n})|^2} \right\}$,  $\sL_1(X_{g,n})\leq \epsilon$ and $c_1(\epsilon)\sqrt{\frac{\sL_1(X_{g,n})}{1000\epsilon}}\leq \frac{1}{64}$ where $c_1(\epsilon)=\frac{14\epsilon c(\epsilon)}{\sqrt{\area(B(\epsilon))}}$, by Proposition \ref{keypartwuxue}, we have $$
\frac{\sigma_1(X_{g,n})}{\sL_1(X_{g,n})}\geq \frac{1}{2^{14}|\area(X_{g,n})|^2}.
 $$

If $\sL_1(X_{g,n})\geq \min\left\{\epsilon, \epsilon^\prime \right\}$ where $\epsilon^\prime=1000\epsilon(\frac{1}{64c_1(\epsilon)})^2$, by Proposition  \ref{keypartepsilon}, we have $$ 
\frac{\sigma_1(X_{g,n})}{\sL_1(X_{g,n})}\geq \frac{\min\left\{\hat{c}_\epsilon(L), \hat{c}_{\epsilon^\prime}(L)\right\}}{|\area(X_{g,n})|^2}.
$$

Recall that by Gauss-Bonnet $\area(X_{g,n})=2\pi(2g+n-2)$. Then the conclusion follows by setting
$$C(L)=\min \left\{\frac{1}{4s(L)}, \frac{1}{2^{14}},\hat{c}_\epsilon(L),\hat{c}_{\epsilon^\prime}(L),\frac{1}{1000\epsilon}\right\}.$$

The proof is complete.
\ep

We enclose this section with the following example showing that the assumption that each boundary component has its length no more than $L$ for some constant $L>0$ in Theorem \ref{mt-2} cannot be dropped. More precisely,
\begin{example}\label{Examp}
Let $X_g$ be a closed hyperbolic surface of genus $g$ as described in \cite{buser1994period} such that the systole $\sys(X_g)\geq U \ln g$ for some uniform constant $U>0$. Remove a point from $X_g$ and consider the unique complete hyperbolic metric corresponding to its complex structure. Then we get a hyperbolic surface $X_{g,1}$ of genus $g$ with $1$ puncture. Schwarz's Lemma implies that the systole 
\[\sys(X_{g,1})\geq \sys(X_g)\geq U\ln g.\] 
According to  Theorem \ref{parlierlength}, there is  a compact hyperbolic surface $\tilde{X}_{g,1}(l)$ of boundary length $$l= 2\ln\ln g$$ satisfying  $$\sys(\tilde{X}_{g,1}(l))\geq \sys(X_{g,1})\geq U \ln g.$$
Let $P$ be a pair of pants with boundary lengths $\left\{4\sinh^{-1}\left(\frac{\cosh(\ln\ln g)}{\sinh{\frac{2}{g^2}}}\right), l ,l\right\}$, and denote the three boundary closed geodesics by $\alpha,\beta,\gamma$ in order. For large $g$, we have $4\sinh^{-1}\left(\frac{\cosh(\ln\ln g)}{\sinh{\frac{2}{g^2}}}\right)\sim 8\ln g$. Glue two copies of $\tilde{X}_{g,1}(l)$'s onto $P$ along $\beta$ and $\gamma$ with the same twist parameters. Then we will get a hyperbolic surface $X_{2g,1}$. 
\begin{figure}[h]
    \centering
    \includegraphics[width=4 in]{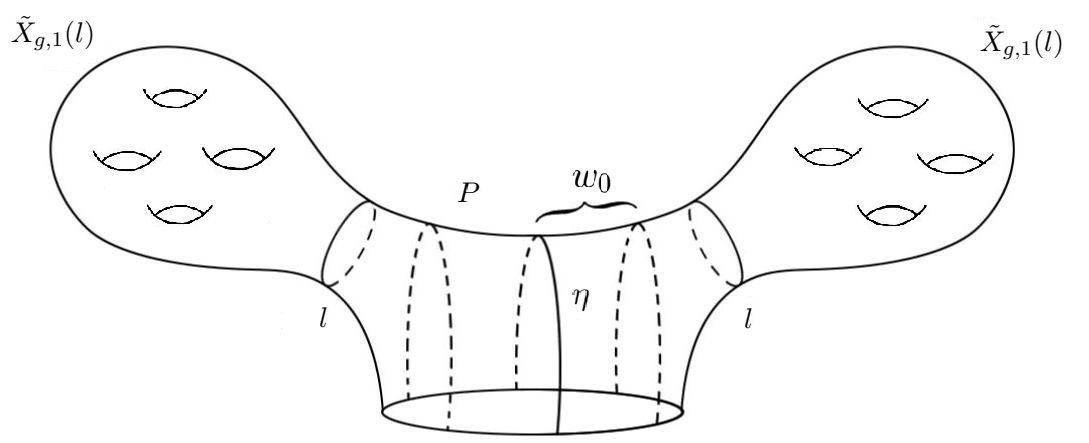}
    \caption{An example with $C(L)\to 0$ as $L\to \infty$}
    \label{fig:example34}
\end{figure}
Take $\eta$ to be the shortest geodesic arc with two ends on $\alpha$, and then cut $P$ along the three perpendiculars we get four congruent right-angled pentagons. A direct computation (see \eg \cite[Formula $2.3.4$ on Page $454$]{buser2010geometry}) shows that 
\[\cosh \left(\frac{l}{2}\right)=\sinh\left(\frac{\ell_\eta(X_{2g,1})}{2}\right)\cdot \sinh \left(\sinh^{-1}\left(\frac{\cosh(\ln\ln g)}{\sinh{\frac{2}{g^2}}}\right) \right)\]
implying that
$$\ell_\eta(X_{2g,1})=\frac{4}{g^2}.$$ 
Then it follows from the Collar Lemma that $\eta$ has a ``half" collar $\mathrm{halC}$ of width $$w_0=\sinh^{-1}\frac{1}{\sinh \ell_\eta(X_{2g,1})}\sim 2\ln g,$$ which is isomorphic to $[0,\frac{1}{2}]\times [-w_0,w_0]$ endowed with the hyperbolic metric $$(2\ell_\eta(X_{2g,1}))^2\cosh^2\rho dt^2+d\rho^2.$$ 
Recall that $\mathcal{L}_1(X_{2g,1})$ is the minimal possible sum of the lengths of simple closed multi-geodesics in $X_{2g,1}$ that cut $X_{2g,1}$ into $2$ components. Now we claim that for large $g>0$, $$\sL_1(X_{2g,1})= l= 2\ln\ln g.$$ (See Figure \ref{fig:example34} for an illustration). Indeed, both $\beta$ and $\gamma$ separate $X_{2g,1}$. Thus, we have $\mathcal{L}_1(X_{2g,1})\leq l=2\ln\ln g.$
Consider a separating simple closed  multi-geodesic
 $\Gamma\subset X_{2g,1}$. If $\Gamma$ transversely intersects with either $\beta$ or $\gamma$, then $\Gamma$ crosses the ``half" collar $\mathrm{halC}$. So we have $\ell_\Gamma(X_{2g,1})\geq 2w_0>l$ for large $g>0$. Otherwise, $\Gamma$ is contained in a single $\tilde{X}_{g,1}(l)$ which in particular implies that $\ell_\Gamma(X_{2g,1})\geq \sys(\tilde{X}_{g,1}(l))>l$
 for large $g>0$.  In summary, we conclude that $\mathcal{L}_1(X_{2g,1})=l=2\ln\ln g$ for large $g$, which is realized by either $\beta$ or $\gamma$.

Set $$f=\frac{\tanh \rho}{\tanh w_0}$$ in this ``half" collar $\mathrm{halC}$, and extend it continuously to  constants  $1$ and $-1$ outside $\mathrm{halC}$. It is clear that 
$$\int_{X_{2g,1}}f=0.$$
Take it as a test function for the Rayleigh quotient, and then a direct computation shows that $|\nabla f|=\frac{1}{\tanh w_0}\frac{1}{\cosh^2\rho}\frac{\partial}{\partial \rho}$ in $\mathrm{halC}$, and $|f|=1$ outside $\mathrm{halC}$. Hence,  
we have
\bear
\int_{X_{2g,1}}|\nabla f|^2&=&\int_{0}^{\frac{1}{2}} \int_{-w_0}^{w_0}\frac{1}{(\tanh w_0)^2}\frac{1}{\cosh^4\rho} 2\ell_\eta(X_{2g,1}) \cosh \rho d\rho dt \nonumber\\
&\leq& \frac{\ell_\eta(X_{2g,1})}{\tanh^2 w_0}\cdot 2\int_{0}^\infty \frac{1}{\cosh^3 \rho}d\rho \prec \ell_\eta(X_{2g,1})=\frac{4}{g^2} \nonumber
\eear
and
$$
\int_{X_{2g,1}}f^2\geq 2 \area(\tilde{X}_{g,1}(\ell))\asymp g.
$$
It follows that 
\bear
\sigma_1(X_{2g,1})\leq \frac{\int_{X_{2g,1}}|\nabla f|^2}{\int_{X_{2g,1}}f^2}\prec \frac{1}{g^3 }. \nonumber
\eear
Therefore, we have $$\frac{g^2\sigma_1(X_{2g,1})}{\sL_1(X_{2g,1})}\prec \frac{1}{g\ln\ln g}\to 0$$ as $g\to\infty$. This shows that the constant $C(L)>0$ in Theorem \ref{mt-2} can not be chosen to be uniform.
\end{example}

\section{Optimal lower bounds for second eigenvalues}\label{lowboundestimation}

In this section we complete the proof of the optimal lower bound in Theorem \ref{mt-1}. More precisely,
\bt[Optimal lower bound] \label{mt-1-lb}
For every $g\geq 3$ and any closed hyperbolic surface $X_g$ of genus $g$,  
\[\lambda_2(X_g) \succ \frac{\mathcal{L}_2(X_g)}{g^2}.\]
Moreover,  for all $g\geq 3$ there exist a closed hyperbolic surfaces $\mathcal{X}_g$ of genus $g$ such that $$\lambda_2(\mathcal{X}_g) \asymp \frac{\mathcal{L}_2(\mathcal{X}_g)}{g^2}.$$
\et

We split the proof into two cases. Firstly based on Theorem \ref{mt-2},  we prove Theorem \ref{mt-1-lb} when $\sL_1(X_g)$ is uniformly bounded from above. That is, 
\begin{proposition}\label{l2sigma2l}
If $X_g$ is a closed hyperbolic surface of genus $g$ with $\sL_1(X_g)<L$ for any fixed constant $L>0$, then there is a constant $E(L)>0$ only depending  on $L$ such that for $g$ large enough, $$
\frac{\lambda_2(X_g)}{\sL_2(X_g)}\geq \frac{E(L)}{g^2}.
$$
\end{proposition}

\bp
Let $\gamma=\cup_{i=1}^k\gamma_i$ be a  separating  simple closed multi-geodesic on $X_g$  of length $\ell_\gamma(X_{g})=\sL_1(X_g)$ and $X_{g}-\gamma=M_1\cup M_2$. By Theorem \ref{maxminpr}, it suffices to show that \be\label{thm1401}
\min\left\{\sigma_1(M_1), \sigma_1(M_2)\right\}\geq \frac{E(L)}{g^2}\sL_2(X_g).\ene
 
\underline{Case-1:} \textit{$|\chi(M_i)|=1$, i.e., $M_i$ is of the type $S_{1,1}$ or $S_{0,3}$.}
 By Lemma \ref{eigens03s11} and  Proposition \ref{boundlk}, we have
 \begin{equation}\label{131}
     \frac{\sigma_1(M_i)}{\sL_2(X_g)}\geq \frac{K(L)}{c_2\ln \left(4\pi(g-1)\right)}\geq \frac{K(L)}{c_2g^2}.
 \end{equation}
 
\underline{Case-2:} \textit{$|\chi(M_i)|>1$.} Since  $\ell_{\partial M_i}(M_i)=\sL_1(X_g)<L$, 
 by Theorem \ref{mt-2}, we have   \be\label{thm1402}
 \sigma_1(M_i)\geq \frac{C(L)}{|\area(M_i)|^2}\sL_1(M_i).\ene
 Consider a separating simple closed multi-geodesic $\eta \subset M_i$ which divides $M_i$ into two components   $A\cup B$ with $\ell_\eta(M_i)=\sL_1(M_i)$ (see Figure \ref{fig:theorem36} for an illustration).
 \begin{figure}[h]
    \centering
    \includegraphics[width=2.5 in]{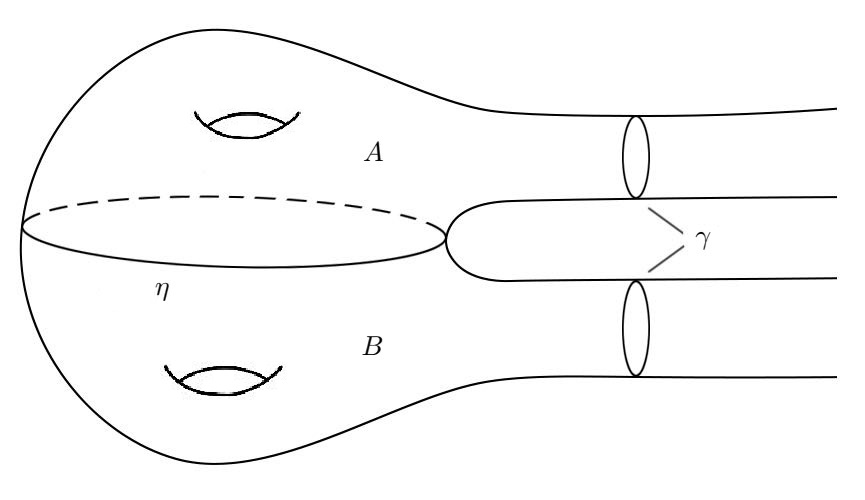}
   \caption{ Comparison of $\sL_1(M_i)$ and $\sL_2(X_g)$  }
    \label{fig:theorem36}
\end{figure}
Since $\gamma\cup\eta$ separates  $X_g$ into $3$ components $\{A,B,X_g-M_i\}$, we have $\sL_1(M_i)+\sL_1(X_g)\geq \sL_2(X_g)$. It is clear that  $\ell_{\partial A}(X_{g})\geq \sL_1(X_g)$ and $\ell_{\partial B}(X_g)\geq \sL_1(X_g)$. Since $\ell_{\partial A}(X_g)+\ell_{\partial B}(X_g)=2\sL_1(M_i)+\sL_1(X_g)$, we have  $$\sL_1(M_i)\geq \frac{1}{2}\sL_1(X_g),$$ and therefore  $$\sL_1(M_i)\geq \frac{1}{3}\sL_2(X_g).$$ Then it follows from (\ref{thm1402}) that  \begin{equation}\label{132}
    \frac{\sigma_1(M_i)}{\sL_2(X_g)}\geq \frac{C(L)}{3|\area(M_i)|^2}\geq \frac{C(L)}{3|\area(X_g)|^2}.
\end{equation}

Combine (\ref{131}) and (\ref{132}), then we finish the proof of (\ref{thm1401})
by choosing 
$E(L)=\min\left\{\frac{K(L)}{c_2},\frac{C(L)}{3\times (4\pi)^2}\right\}$.
\ep

Now we prove Theorem \ref{mt-1-lb} when $\sL_1(X_g)$ is uniformly bounded away from $0$. That is, 
\begin{proposition}\label{l2sigma2l-2}
For any fixed constant $L>0$, if $X_g$ is a closed hyperbolic surface of genus $g$ with $\sL_1(X_g)\geq L$, then there is a constant $E^\prime(L)>0$ only depending  on $L$ such that for $g$ large enough, $$
\frac{\lambda_2(X_g)}{\sL_2(X_g)}\geq \frac{E^\prime(L)}{g^2}.
$$
\end{proposition}

\bp
Similar to the proof of the previous proposition, we take a separating simple closed multi-geodesic $\gamma\subset X_g$ that cuts $X_g$ into $M_1\cup M_2$ with $\sL_1(X_g)=\ell_\gamma(X_g)$. By Theorem \ref{maxminpr}, it suffices to show that \be\label{thm1401-1}
\min\left\{\sigma_1(M_1), \sigma_1(M_2)\right\}\geq \frac{E^\prime(L)}{g^2}\sL_2(X_g).\ene

Now we prove it case by case.

\underline{Case-1:}  \textit{$|\chi(M_i)|=1$, i.e., $M_i$ is of the type $S_{1,1}$ or $S_{0,3}$.} In this case, if  $\ell_{\partial M_i}(X_g)=\sL_1(X_g)\geq \ln g$, by Proposition
\ref{boundlk} and Theorem \ref{mt-WX-1}, we have
\begin{equation}\label{geqlng}
    \lambda_2(X_g)\geq \lambda_1(X_g)\succ \frac{\sL_1(X_g)}{g^2}\geq\frac{\ln g}{g^2}\succ \frac{\sL_2(X_g)}{g^2}.
\end{equation}
While if $\ell_{\partial M_i}(X_g)\leq \ln g$, by Lemma \ref{s03s11lngsigma}, we have $$\sigma_1(M_i)\succ \frac{1}{g}.$$ Then combined with Proposition
\ref{boundlk}, it follows that for large enough $g$
 \begin{equation}\label{leqlng}
    \sigma_1(M_i)\succ \frac{\sL_2(X_g)}{g^2}.
\end{equation}

\underline{Case-2:}  \textit{$|\chi(M_i)|>1$ and the width $w(\eta)$ of the maximal half collar of some boundary component $\eta$ in $M_i$ is shorter than $\sL_1(X_g)$.} In this case, we use the same argument as in the proof of Proposition \ref{boundlk}. If the closure of the maximal half collar intersects with another boundary component $\eta^\prime$, then there will be a simple closed geodesic $\delta$ of length shorter than $\ell_\eta(X_g)+\ell_{\eta^\prime}(X_g)+2w(\eta)$, which bounds a pair of pants along with $\{\eta,\eta^\prime\}$ in $M_i$. By our assumption that $w(\eta)<\sL_1(g)$, we have $$\sL_2(X_g)\leq \ell_\delta(X_g)+\sL_1(X_g)\leq 4 \sL_1(X_g).$$ Then it follows from Theorem \ref{mt-WX-1} that
\begin{equation}\label{geq1short1}
    \lambda_2(X_g)\geq \lambda_1(X_g)\succ \frac{\sL_1(X_g)}{g^2} \geq \frac{\sL_2(X_g)}{4g^2}.
\end{equation}
If the closure of  the  maximal half collar  does not meet another boundary component, there  will be two simple closed geodesics $\delta_1$ and $\delta_2$ of lengths shorter than $\ell_\eta(X_g)+2w(\eta)$, which will bound a pair of pants along with $\eta$ in $M_i$. Again by our assumption that $w(\eta)<\sL_1(g)$, in this case we have $$\sL_2(X_g)\leq \ell_{\delta_1}(X_g)+\ell_{\delta_2}(X_g)+\sL_1(X_g)\leq 7\sL_1(X_g).$$ Also by Theorem \ref{mt-WX-1} we have  
 \begin{equation}\label{geq1short2}
    \lambda_2(X_g)\geq \lambda_1(X_g)\succ \frac{\sL_1(X_g)}{g^2} \geq \frac{\sL_2(X_g)}{7g^2}.
\end{equation}

\underline{Case-3:}  \textit{$|\chi(M_i)|>1$ and the width of the maximal half collar of  each boundary component in $M_i$ is greater than or equal to $\sL_1(X_g)$.} In this case  we use Cheeger's inequality to estimate $\sigma_1(M_i)$. Take a set of piecewise smooth curves $\Gamma$ that divides $M_i$ into two components $A$ and $B$ with $$\frac{\ell(\Gamma)}{\min\{\area(A),\area(B)\}}\leq 2h(M_i).$$ If $\Gamma$ only contains  curves that bound disks or cylinders, then by Lemma \ref{isoperimetric} we have $\frac{\ell(\Gamma)}{\min\{\area(A),\area(B)\}}\geq \frac{1}{2}$. Then according to Theorem \ref{cheineq} it follows that  $\sigma_1(M_i)\geq \frac{1}{64}.$ Therefore by Proposition \ref{boundlk}, we have
\begin{equation}\label{geqepsilon1}
    \sigma_1(M_i)\geq \frac{\sL_2(X_g)}{64 c_2\ln \left(4\pi(g-1)\right)}\geq \frac{\sL_2(X_g)}{64c_2g^2}.
\end{equation}
If $\Gamma$  contains a curve with two endpoints on $\partial M_i$ which is not homotopic to an arc on one boundary component of $M_i$, then $$\ell(\Gamma)\geq \min\{w(\eta),\eta\subset \partial M_i\}$$
where $w(\eta)$ represents the width of the maximal half collar of $\eta$ in $M_i$. Assume the minimum is taken at $\eta_0$ and $w_0=w(\eta_0)$. By our assumption we know that $w_0\geq \sL_1(X_g)\geq L$. Then it follows from Cheeger's inequality \ie, Theorem \ref{cheineq} that
\begin{equation}\label{geqepsilon2}
    \sigma_1(M_i)\geq\frac{1}{16}\frac{\ell^2(\Gamma)}{\min\{\area(A),\area(B)\}^2}\geq \frac{L}{16^2\pi^2}\frac{w_0}{g^2}.
\end{equation}
By the same  argument as we mention the maximal half  collar in \underline{Case-2}, we have $\sL_1(M_i)\leq 4w(\eta_0)+2\ell_{\eta_0}(X_g)$. So by our assumption that $w_0\geq \sL_1(X_g)=\ell_{\partial M_i}(X_g)$, we have  \begin{equation}\label{geqepsilon3}
    \sL_2(X_g)\leq \sL_1(M_i)+\ell_{\partial M_i}(X_g)\leq 7w_0.
\end{equation}
Then it follows from (\ref{geqepsilon2}) and (\ref{geqepsilon3}) that 
\begin{equation}\label{geqepsilon4}
\sigma_1(M_i)\geq \frac{L}{7\cdot16^2\pi^2}\frac{\sL_2(X_g)}{g^2}.
\end{equation}
For the remaining case, \ie, $\Gamma$ contains no curve bounding a disk or cylinder and no curve with two endpoints on $\partial M_i$ which is not homotopic to an arc on certain component, since $\Gamma$ separates $M_i$, in this case we have $\ell(\Gamma)\geq \sL_1(M_i)$. So
by Cheeger's inequality \ie, Theorem \ref{cheineq}, we have \begin{equation}\label{geqepsilon5}
\sigma_1(M_i)\geq \frac{\ell^2(\Gamma)}{16\pi^2g^2}\geq \frac{\sL_1(M_i)\cdot \sL_1(M_i)}{16\pi^2g^2} \geq \frac{L}{96\pi^2}\frac{\sL_2(X_g)}{g^2}.
\end{equation}
Here in the last inequality we apply $\sL_1(M_i)\geq \frac{1}{2}\sL_1(X_g)\geq\frac{L}{2}$ and $\sL_1(M_i)\geq \frac{1}{3}\sL_2(X_g)$ which have been shown in the proof of Theorem \ref{l2sigma2l}.

Then, the conclusion follows from \eqref{thm1401-1}, \eqref{geqlng}, \eqref{leqlng}, \eqref{geq1short1}, \eqref{geq1short2}, \eqref{geqepsilon1},
\eqref{geqepsilon4}, and \eqref{geqepsilon5}, by a suitable choice of $E^\prime(L)$ only depending on $L$.
\ep

Now we are ready to prove Theorem \ref{mt-1-lb}.

\bp[Proof of Theorem \ref{mt-1-lb}]
Take $L=1$. Then it follows from Proposition \ref{l2sigma2l} and Proposition \ref{l2sigma2l-2} that
\[\lambda_2(X_g)\geq \min\{E(1), E^\prime(1)\}\cdot \frac{\mathcal{L}_2(X_g)}{g^2}.\]

The existence of $\mathcal X_g$ with $\frac{\lambda_2(\mathcal X_g)}{\sL_2(\mathcal X_g)}\asymp \frac{1}{g^2}$ was constructed in \cite{WX18}. We only briefly introduce it as follows. One may see \cite{WX18} for more details. Let $\mathcal{P}_\ell$ be the pair of pants $\mathcal{P}_\ell$ whose boundary curves all have length equal to $\ell$ where $\ell<\arcsinh 1$. Then we glue $(2g-2)$ copies of $\mathcal P_\ell$'s from left to right (see \cite[Proposition 3]{WX18}) to get a closed hyperbolic surface $\mathcal X_g$ of genus $g$. For any closed geodesic $\gamma \subset \mathcal X_g$, the curve $\gamma$ is either one of the boundary closed geodesic of one $\mathcal{P}_\ell$ or must intersect at least one of the boundary closed geodesic of one $\mathcal{P}_\ell$. Then the Collar Lemma implies that $\ell(\gamma)> 2\arcsinh 1$ if $\gamma$ intersects with one of the boundary closed geodesic of one $\mathcal{P}_\ell$. This in particular yields that 
\begin{equation}
\mathcal \sL_1(\mathcal X_g)=\ell \ \textit{and} \ \sL_2(X_g)= 2\ell. \nonumber
\end{equation}

\noindent It was proved in \cite[Proposition 11]{WX18}  that
\begin{equation}
\lambda_2(\mathcal X_g) \leq \frac{\beta(\ell)}{g^2} 
\end{equation}
\noindent where for the case that $\ell< 2\arcsinh 1$ and $\mathcal X_g$ is constructed as above, $$\beta(\ell) \asymp \ell.$$ So we have
\begin{equation}
\lambda_2(\mathcal X_g) \prec \frac{\mathcal L_2(\mathcal X_g)}{g^2} 
\end{equation}
which together with the lower bound implies that $$\frac{\lambda_2(\mathcal X_g)}{\sL_2(\mathcal X_g)}\asymp \frac{1}{g^2}.$$

The proof is complete.
\ep

\begin{rem*}
For a general index $k\geq 2$ independent of $g$, we only need to replace a multi-curve $\gamma$ realizing $\sL_1(X_g)$  by a multi-curve realizing $\sL_{k-1}(X_g)$ in the argument above to obtain that
\be
\lambda_k(X_g)\succ \frac{\mathcal{L}_k(X_g)}{g^2}.\nonumber
\ene
And the surface $\mathcal X_g$ in \cite[Proposition 11]{WX18} also satisfies that 
\[\lambda_k(\mathcal X_g)\asymp \frac{\mathcal{L}_k(\mathcal X_g)}{g^2}.\] 
\end{rem*}

\section{Optimal upper bounds for second eigenvalues}\label{upperbound}
In this section we complete the proof of the optimal upper bound in Theorem \ref{mt-1}. More precisely,
\bt[Optimal upper bound] \label{mt-1-up}
For every $g\geq 3$ and any closed hyperbolic surface $X_g$ of genus $g$,  
\[\lambda_2(X_g) \prec \mathcal{L}_2(X_g).\]
Moreover,  for all $g\geq 3$ there exists a closed hyperbolic surfaces $\mathcal{Y}_g$ of genus $g$ such that $$\lambda_2(\mathcal{Y}_g) \asymp \mathcal{L}_2(\mathcal{Y}_g).$$
\et

It is known by Cheng \cite{cheng1975eigenvalue} that the eigenvalues of $X_g$ can be bounded from above by the diameter $\diam(X_g)$ of $X_g$. More precisely,
for all $k\geq 1$,
\be\label{Cheng}
   \lambda_k(X_g)\leq \frac{1}{4}+\frac{16\pi^2\cdot k^2}{\diam(X_g)^2}.
\ene

\noindent A standard area argument together with Gauss-Bonnet implies that 
\[\diam(X_g)\geq \ln (4g-3).\]
The two inequalities above yield that for all $k=o(\ln g)$,
\be\label{1/4-ub-ei}
\limsup \limits_{g\to \infty}\sup_{X_g \in \sM_g}\lambda_k(X_g)\leq \frac{1}{4}.
\ene
A recent breakthrough of Hide-Magee \cite{HM21} says that this upper bound $\frac{1}{4}$ is optimal. More precisely, they showed that for all $k=o(\ln g)$,
\[\lim \limits_{g\to \infty}\sup_{X_g \in \sM_g}\lambda_k(X_g)= \frac{1}{4}.\]
One may also see \cite{WZZ22} for optimal higher spectral gaps of closed hyperbolic surfaces of large genus. In this article we study its connection to the geometric quantity $\sL_2(X_g)$. First we refine the argument in \cite{SWY80} to show the upper bound in Theorem \ref{mt-1-up}.

\bp[Proof of Part $(1)$ of Theorem \ref{mt-1-up}]
We split the proof into two cases.

\underline{Case-1:} \textit{$\sL_2(X_g)>2\arcsinh 1.$} Recall that $\diam(X_g)\geq \ln (4g-3)\geq \ln 9$. Then it follows from \eqref{Cheng} that 
\begin{equation}\label{partA}
\lambda_2(X_g)\leq \frac{1}{2\arcsinh 1}
\left(\frac{1}{4}+\frac{64\pi^2}{(\ln{9})^2}\right)\cdot \sL_2(X_g)\asymp \sL_2(X_g).
\end{equation}

\underline{Case-2:} \textit{$\sL_2(X_g)\leq 2\arcsinh 1.$} Assume that $\sL_2(X_g)$ is realized by a simple closed multi-geodesic $\gamma=\sum_{i=1}^m\gamma_i$. We write $X_g-\cup_{i=1}^m\gamma_i=\cup_{j=1}^{3}M_j$ and $\partial M_j=\cup_{t=1}^{s_j}\gamma_{j_t}$. First by our assumption on $\sL_2(X_g)$ and the Collar Lemma (see Lemma \ref{collar lemma}), the half width $w(\gamma_i)$ of each collar $T(\gamma_i)$ satisfies $$w(\gamma_i)\geq \arcsinh1.$$
Now define test functions $\psi_j\in W^{1,2}(X_g)$ for each $1\leq j\leq 3$ through
$$
\begin{aligned}
\psi_j(p)= \left \{
\begin{array}{ll}
   1,                    &p\in M_j-\cup_{t=1}^{s_j}T(\gamma_{j_t});\\
   0,         & p\in X_g-M_j-\cup_{t=1}^{s_j}T(\gamma_{j_t});\\
    \frac{1}{2}+\frac{\tanh\rho }{2\tanh w(\gamma_{j_t})},                               & p\in T(\gamma_{j_t})\cong [-w(\gamma_{j_t}),w(\gamma_{j_t})]\times \R/\Z.
\end{array}
\right.
\end{aligned}
$$

\noindent A direct computation shows that for each $1\leq j \leq 3$, 
\[\int_{X_g}\psi_j^2\geq \int_{M_j-\cup_{t=1}^{s_j}T(\gamma_{j_t})} \psi_j^2= \area( M_j-\cup_{t=1}^{s_j}T(\gamma_{j_t}))\succ 1,\]
and on each collar $T(\gamma_i)$
\begin{equation*}
    \begin{aligned}
         \int_{T(\gamma_i)}|\nabla\psi_j|^2&=\frac{\ell_{\gamma_i}(X_g)}{4 (\tanh w(\gamma_i))^2}\int_{-w(\gamma_i)}^{w(\gamma_i)}\frac{1}{\cosh^3\rho}d\rho\\
         \leq &\frac{\ell_{\gamma_i}(X_g)}{2 (\tanh \arcsinh 1)^2}\int_{0}^\infty\frac{1}{\cosh^3\rho}d\rho\asymp \ell_{\gamma_i}(X_g).
        \end{aligned}
\end{equation*}

\noindent Clearly $\psi_1,\psi_2,\psi_{3}$ are linearly independent. So there is a nonzero linear combination $\psi$ of them such that $\left<\psi,\phi_k\right>=0$ for $k=0,1$ where $\phi_k$ is the $k$-th eigenfunction of $X_g$. Since the supports of $\psi_j$ are pairwise disjoint, we have 

\begin{equation}\label{partB}\lambda_2(X_g)\leq \frac{\int_{X_g}|\nabla\psi|^2}{\int_{X_g}\psi^2}\prec \sum \limits_{i=1}^m \ell_{\gamma_i}(X_g)=  \sL_2(X_g).\end{equation}

Then the conclusion follows from \eqref{partA} and \eqref{partB}.
\ep

\begin{rem*}
The proof above actually shows that
\begin{proposition}
Let $X_g$ be a closed hyperbolic surface of genus $g$, then for any $k\in [1,2g-3]$ with $k\prec \diam(X_g)$, we have
\[\lambda_k(X_g)\prec \sL_k(X_g).\]
In particular, for any $k\in [1,2g-3]$ with $k\prec \ln g$, we have
\[\lambda_k(X_g)\prec \sL_k(X_g).\]
\end{proposition}
\end{rem*}

Now we prove that the upper bound in Theorem \ref{mt-1-up} is optimal in the sense that for all $g\geq 3$, there exists a closed hyperbolic surface $\mathcal Y_g$ of genus $g$ such that $\lambda_2(\mathcal Y_g)\asymp \sL_2(\mathcal Y_g)$. It suffices to consider cases for large $g$. The construction is based on recent results on Weil-Petersson random surfaces. By \cite[Theorem 4.2, Theorem 4.5, Theorem 4.8]{mirzakhani2013growth} we know that
\be\label{mirz-inj}
\liminf\limits_{g\to \infty}\Prob\left(X_g\in \sM_g; \ \sys(X_g)\asymp 1\right)>0,
\ene 
\be\label{mirz-inr}
\lim\limits_{g\to \infty}\Prob\left(X_g\in \sM_g; \ \max_{p \in X_g}\mathrm{inj}(p)\asymp \ln g\right)=1
\ene
where $\mathrm{inj}(p)$ is the injectivity radius of $X_g$ at $p$, and
\be\label{mirz-ch}
\lim\limits_{g\to \infty}\Prob\left(X_g\in \sM_g; \ h(X_g)\asymp 1\right)=1
\ene
All three results above have more delicate statements in \cite{mirzakhani2013growth}. 

\begin{con*}[for $\mathcal Y_g$] We construct $\sY_g$ by using the following four steps (see Figure \ref{fig:exop} for an illustration):
\ben
\item[(step-1)] for all large enough $g$, firstly by \eqref{mirz-inj}, \eqref{mirz-inr} and \eqref{mirz-ch} we know that there exists a closed hyperbolic surface $X_{g-2}$ of genus $(g-2)$ such that 
\[\sys(X_{g-2})\asymp 1, \ \max_{p \in X_g}\mathrm{inj}(p)\asymp \ln g \  \textit{and} \ h(X_{g-2})\asymp 1;\] 
\item[(step-2)] let $p_0\in X_{g-2}$ such that $\inj(p)\asymp \ln g$. Then we remove $p_0$ from $X_{g-2}$ and let $\bar{X}_{g-2,1}$ be the unique complete hyperbolic surface corresponding to the complex structure on $X_{g-2}\setminus \{p_0\}$ induced from $X_{g-2}$;
\item[(step-3)] by  Theorem \ref{parlierlength} we let $X_{g-2,1}(1) \in \sM_{g-2,1}(1)$ be a compact hyperbolic surface of genus $(g-2)$ with one geodesic boundary of length $1$ such that for all simple closed curve $\gamma$, $\ell_\gamma(X_{g-2,1}(1))\geq \ell_\gamma(\bar{X}_{g-2,1})$. To simplify notation, we use $X_{g-2,1}$ for $X_{g-2,1}(1)$;
\item[(step-4)] let $S_{2,1}\in \sM_{2,1}(1)$ be any fixed compact hyperbolic surface of genus $2$ with one geodesic boundary of length $1$. The desired closed hyperbolic $\sY_g$ is obtained by gluing $X_{g-2,1}$ and $S_{2,1}$ along their boundaries by any twist parameter.  
\een
\end{con*}

\begin{figure}[h]
    \centering
    \includegraphics[width=6.4 in]{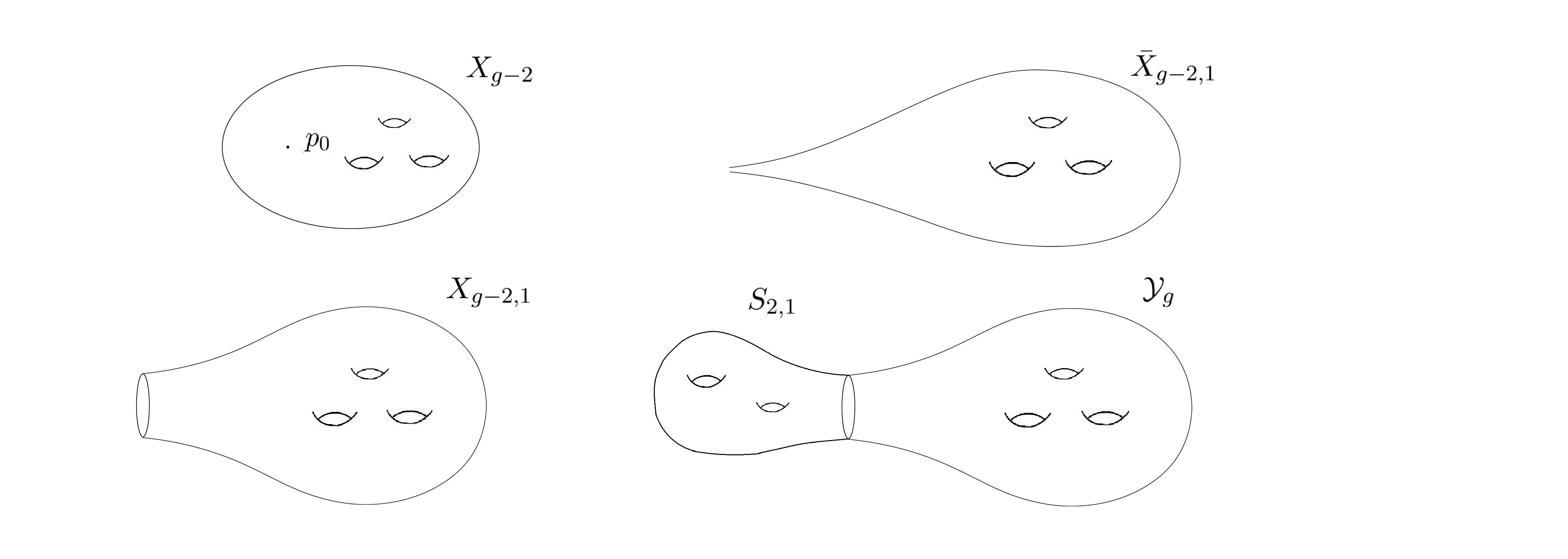}
    \caption{ Construction of $\sY_g$}
    \label{fig:exop}
\end{figure}

Our aim is to show that
\be
\lambda_2(\sY_g)\asymp \sL_2(\sY_g)\asymp 1.
\ene

\noindent By construction,  \[\sL_2(\mathcal{Y}_g)\leq 1+ \sL_1(S_{2,1})\prec 1.\]

\noindent We have shown that $\lambda_2(\sY_g)\prec \sL_2(\sY_g)$. So by Cheeger's inequality, \ie, Theorem \ref{cheineq}, it remains to show that
\be\label{up-e-a}
h(\mathcal{Y}_g)\succ 1.
\ene 

Before showing \eqref{up-e-a}, by the construction of $\sY_g$ we make the following two useful observations:
\ben
\item since $\inj(p_0)\asymp \ln g$, for large $g$ the cusped hyperbolic surface $\bar{X}_{g-2,1}$ satisfies the so-called large cusp condition in \cite{brooks1999platonic}. By \cite[Theorem 4.1]{brooks1999platonic} we know that
\be\label{eq-e-h}
h(\bar{X}_{g-2,1})\asymp h(X_{g-2})\asymp 1.
\ene

\item By Schwarz Lemma we know that $\sys(\bar{X}_{g-2,1})\geq \sys(X_{g-2})\asymp 1$. So it follows from  Theorem \ref{parlierlength} that  
\be\label{eq-e-sys}
\sys(X_{g-2,1})\geq \sys(\bar{X}_{g-2,1})\succ 1.
\ene
\een

First we show that
\bl \label{uniformgapxg-n1}
For large $g>0$,
\[h(X_{g-2,1})\succ 1.\]
\el
\bp
By Lemma \ref{collar lemma} the boundary component of $X_{g-2,1}$ contains a collar $T$ of width $w_0=\arcsinh \frac{1}{\sinh 0.5}.$ By \cite[Remark 3.5]{HM96}, each component of the set of curves realizing $h(X_{g-2,1})$ is embedded. By small permutation, one may assume that $\Gamma$, consisting of \emph{simple} curves, separates $X_{g-2,1}$ into two parts $A$ and $B$ with  $\partial A \cap \mathring{X}_{g-2,1}=\partial B\cap \mathring{X}_{g-2,1}=\mathring{\Gamma}$ such that
\be \label{in-l-9}
h(X_{g-2,1})\geq\frac{1}{2}\frac{\ell(\Gamma)}{\min\{\area(A),\area(B)\}}=\frac{1}{2}\max\left\{\frac{\ell(\partial A)}{\area(A)},\frac{\ell(\partial B)}{\area(B)}\right\}.
\ene Now we prove the claim case by case.  

\underline{Case-1:}   \textit{both $A$ and $B$ contain only disks or cylinders.} For this case it follows from Lemma \ref{isoperimetric} that \be\label{xg-n1case1}
\frac{\ell(\Gamma)}{\min\{\area(A),\area(B)\}}\geq \frac{1}{2}.\ene

\underline{Case-2:}  \textit{either $A$ or $B$ does not contain only disks and cylinders, and $$\min\{\area(A),\area(B)\}\leq 16\pi.$$} In this  case, either $\Gamma$ crosses  the half collar $T$  or contains a homotopically non-trivial loop in $X_{g-2,1}$. By \eqref{eq-e-sys} we have
\[\ell(\Gamma)\geq \min\{w_0, \sys(X_{g-2,1})\}\succ 1\]
implying that 
\be\label{xg-n1case2}
\frac{\ell(\Gamma)}{\min\{\area(A),\area(B)\}}\geq \frac{\ell(\Gamma)}{16\pi}\succ 1.\ene

\underline{Case-3:}  \textit{either $A$ or $B$ does not contain only disks and cylinders, and $$\min\{\area(A),\area(B)\}\geq 16\pi.$$} 
\noindent First we recall the construction of Step-3 above in \cite{parlier2005lengths}: take two simple closed geodesics $\{\alpha,\beta\}\subset \bar{X}_{g-2,1}$  which along with the cusp bound a pair of pants $\tilde{P}$; then replace $\tilde{P}$ by $P$,  a pair of  pants with boundary lengths $\{1, \ell_\alpha(\bar{X}_{g-2,1}), \ell_{\beta}(\bar{X}_{g-2,1})\}$, and the desired surface $X_{g-2,1}$ is obtained by gluing $P$ back to $\bar{X}_{g-2,1}\setminus \tilde{P}$ along $\alpha$ and $\beta$ with the unchanged twist parameters. It is shown in \cite[Lemma 3.1]{parlier2005lengths} that for  any  simple  arc $\gamma\subset P$ with  two endpoints on $\alpha\cup\beta$,
$$\ell_{\gamma}(X_{g-2,1})\geq \ell_{\tilde{\gamma}}(\bar{X}_{g-2,1})$$
where $\tilde{\gamma}\subset \tilde{P}$ is the simple geodesic arc with the same endpoints and homotopy type as $\gamma \subset P$. Now we start to prove this case. Similar to Case-2, we also have $\ell(\Gamma)\geq \min\{w_0, \sys(X_{g-2,1})\}\succ 1$. Recall that for the half collar $T$, we always have $\area(T)\leq 2\pi$ and $\ell(\partial T)\asymp 1$. Clearly $\partial(A\cap(X_{g-2,1}\setminus T))\subset (\partial A)\cup(\partial T)=\Gamma \cup (\partial T)$. This gives that
\[\ell(\partial(A\cap(X_{g-2,1}\setminus T))\leq \ell(\partial A)+\ell(\partial T)\prec \ell(\partial A).\]
Similarly,
\[\ell(\partial(B\cap(X_{g-2,1}\setminus T)) \prec \ell(\partial B).\] 
Since $\area(T)\leq 2\pi$ and $\min\{\area(A),\area(B)\}\geq 16\pi$, 
\[\area(A\cap(X_{g-2,1}\setminus T))\asymp \area(A) \ \textit{and} \ \area(B\cap(X_{g-2,1}\setminus T))\asymp \area(B).\]
So we have
\begin{equation}\label{dealintersectboundary}
\begin{aligned}
&\max \left\{\frac{\ell(\partial(A\cap(X_{g-2,1}\setminus T)) )}{\area(A\cap (X_{g-2,1}\setminus T))},\frac{\ell(\partial(B\cap(X_{g-2,1}\setminus T)))}{\area(B\cap (X_{g-2,1}\setminus T))}\right\} \\
&\prec \max\left\{\frac{\ell(\partial A)}{\area(A)},\frac{\ell(\partial B)}{\area(B)}\right\}\leq 2h(X_{g-2,1}).
\end{aligned}
\end{equation}
It suffices  to give a uniform positive lower bound  for the left hand side of \eqref{dealintersectboundary}.
\begin{figure}[h]
    \centering
    \includegraphics[width=3 in]{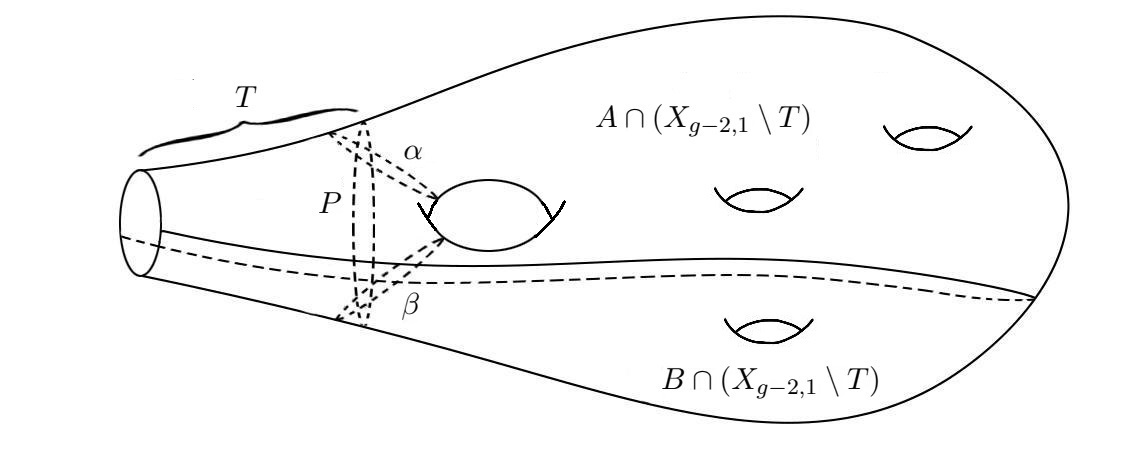}
    \caption{Replaced by interior domains}
    \label{fig:prop721}
\end{figure}
$\mathrm{WLOG}$ we assume that $\area(A\cap (X_{g-2,1}\setminus T))\leq \area(B\cap (X_{g-2,1}\setminus T))$. Our aim is to show that
\be
\frac{\ell(\partial(A\cap(X_{g-2,1}\setminus T)) )}{\area(A\cap (X_{g-2,1}\setminus T))} \succ 1.
\ene
Recall that by our choice, $\Gamma$ consists of simple curves, so each component of $\partial(A\cap(X_{g-2,1}\setminus T))$ is also simple. For sub-arcs of $\partial(A\cap(X_{g-2,1}\setminus T))$ in the pair of pants $P$ with endpoints on $\alpha\cup \beta$, we replace each of them by  the  corresponding simple geodesic arc in $\tilde{P}$ to make them shorter; for sub-arcs of $\partial(A\cap(X_{g-2,1}\setminus T))$ in $X_{g-1}\setminus P$, we keep them to be invariant. Then as introduced above, by \cite[Lemma 3.1]{parlier2005lengths} we will get a set $C\subset \bar{X}_{g-2,1}$ of simple curves which has length $\ell(C)\leq \ell(\partial (A\cap (X_{g-2,1}\setminus T)))$, and  bounds a domain $\tilde{A}\subset \bar{X}_{g-2,1}$ that coincides with $A\cap(X_{g-2,1}\setminus P)$ outside $\tilde{P}$. Since $\area(P)= 2\pi$, $\area(T)\leq 2\pi$ and $\min\{\area(A),\area(B)\}\geq 16\pi$, we have
\[\area(A\cap (X_{g-2,1}\setminus T))-\area(P)\asymp \area(A\cap (X_{g-2,1}\setminus T))\] 
which implies that
$$
\frac{\ell(C)}{\area(\tilde{A})}\leq \frac{\ell(\partial (A\cap (X_{g-2,1}\setminus T)))}{\area(A\cap (X_{g-2,1}\setminus T))-\area(P)}\prec \frac{\ell(\partial (A\cap (X_{g-2,1}\setminus T)))}{\area(A\cap (X_{g-2,1}\setminus T))}.
$$
Similarly, we have
\[\area(\bar{X}_{g-2,1})-\area(\tilde{A})\asymp \area(B\cap (X_{g-2,1}\setminus T))\]
which implies that
$$
 \frac{\ell(C)}{\area(\bar{X}_{g-2,1})-\area(\tilde{A})}\prec \frac{\ell(\partial (A\cap (X_{g-2,1}\setminus T)))}{\area(B\cap (X_{g-2,1}\setminus T))}\leq  \frac{\ell(\partial (A\cap (X_{g-2,1}\setminus T)))}{\area(A\cap (X_{g-2,1}\setminus T))}.
 $$
Thus, it follows from \eqref{eq-e-h} that  
\be\label{interbound}
\begin{aligned}
 \frac{\ell(\partial (A\cap (X_{g-2,1}\setminus T)))}{\area(A\cap (X_{g-2,1}\setminus T))}&\succ \max\left\{\frac{\ell(C)}{\area(\tilde{A})},  \frac{\ell(C)}{\area(\bar{X}_{g-2,1})-\area(\tilde{A})}\right\}\\
&\geq h(\bar{X}_{g-2,1})\asymp 1.
\end{aligned}
\ene
Combining (\ref{dealintersectboundary}) and (\ref{interbound}), we have 
\be\label{xg-n1case3-1}
h(X_{g-2,1})\succ 1.
\ene

Then the conclusion follows from \eqref{in-l-9}, \eqref{xg-n1case1}, \eqref{xg-n1case2} and \eqref{xg-n1case3-1}.
\ep

Now we are ready to show the surface $\mathcal{Y}_g$ is the desired surface in Theorem \ref{mt-1-up}.
\bp[Proof of Part $(2)$ of Theorem \ref{mt-1-up}]
As discussed before, it suffices to show \eqref{up-e-a}, \ie, $$h(\sY_g)\succ 1.$$ Let $\Gamma$ be the set of curves realizing the Cheeger constant $h(\mathcal{Y}_g)$ of $\mathcal{Y}_g$. See \cite{cheegerconstant} for the existence and description of  $\Gamma$. Assume that $\Gamma$ separates $\mathcal{Y}_g$ into two pieces $A\cup B$, with $\Gamma=\partial A=\partial B=A\cap B$ and   $\area(A)\leq \area(B)$. Now we split the proof into the following several cases.

\underline{Case-1:}  \textit{$A\cap X_{g-2,1}=\emptyset.$}
Set $w_0=\min\{\sys(S_{2,1}),w(\gamma)\}$ where $\sys(S_{2,1})$ is the length of shortest closed geodesic in $S_{2,1}$ and $w(\gamma)$ is the width of collar centered at $\gamma$ given by Lemma \ref{collar lemma}. If $A$ contains either disks or cylinders, by Lemma \ref{isoperimetric} we have
\[\frac{\ell(\Gamma)}{\area(A)}\geq \frac{1}{2}.\]
Otherwise, we have $\ell(\Gamma)\geq w_0$ giving that
\[\frac{\ell(\Gamma)}{\area(A)}\geq \frac{w_0}{\area(S_{2,1})}.\]
So we have \be\label{hxg-1}
\frac{\ell(\Gamma)}{\area(A)}\geq \min\left\{\frac{1}{2},\frac{w_0}{6\pi}\right\}\asymp 1.
\ene

\underline{Case-2:}  \textit{$A\subset X_{g-2,1}$.} Since $\area(B)\geq \frac{\area(\mathcal{Y}_g)}{2}=2\pi(g-1)$, $$\area(B\cap X_{g-2,1})\asymp \area(B)\geq \area(A).$$ So we have
\be\label{hxg-2}
\begin{aligned}
\frac{\ell(\Gamma)}{\area(A)}&\succ  \max \left\{\frac{\ell(\Gamma)}{\area(A\cap X_{g-2,1})},\frac{\ell(\Gamma)}{\area(B\cap X_{g-2,1})}\right\}\\
&\geq h(X_{g-2,1})\succ 1. 
\end{aligned}
\ene

\underline{Case-3:} \textit{$A\cap X_{g-2,1}\neq\emptyset$ and  $A\cap S_{2,1}\neq \emptyset$.} For this case, we have two sub-cases.
\underline{Case-3-(a):} \textit{either $\Gamma\cap \gamma=\emptyset$ or $\gamma\subset \Gamma$.} By the same argument as in Case-1 and Case-2 we have 
\be\label{hxg-3}
\frac{\ell(\Gamma)}{\area(A)}\succ \min\left\{\frac{1}{2},\frac{w_0}{6\pi}, h(X_{g-2,1})\right\}\succ 1.
\ene
\noindent 
\underline{Case-3-(b):}  \textit{$\Gamma\cap\gamma\neq \emptyset$, and $\gamma\not\subset\Gamma$.} In this case, firstly by \cite{cheegerconstant} it is known that $\gamma$ transversely intersects with $\Gamma$.  If $A$ contains disks or cylinders, it follows from Lemma  \ref{isoperimetric} that  \be\label{hxg-4-1}
\frac{\ell(\Gamma)}{\area(A)}\succ 1.
\ene
Otherwise  $\ell(\Gamma)\geq w_0$. If $\area(A\cap X_{g-2,1})\leq 16\pi$, then $\area(A)\leq\area(A\cap X_{g-2,1})+\area(S_{2,1}) \leq 22\pi$. So we have \be\label{hxg-4-2}
\frac{\ell(\Gamma)}{\area(A)}\geq \frac{w_0}{22\pi}\asymp 1.
\ene
If $\area(X_{g-2,1}\cap A)\geq 16\pi$, since $\area(\sY_g\setminus X_{g-2,1})=\area(S_{2,1})=6\pi$, it is not hard to see that
\[\area(A)\asymp \area(A\cap X_{g-2,1})\prec \area(B\cap X_{g-2,1}).\]
So we have \be\label{hxg-4-3}
\begin{aligned}
\frac{\ell(\Gamma)}{\area(A)}&\succ\frac{\ell(\Gamma\cap X_{g-2,1})}{A \cap \area(X_{g-2,1})}\\
\succ&\max \left\{\frac{\ell(\Gamma\cap X_{g-2,1})}{\area(A\cap X_{g-2,1})},\frac{\ell(\Gamma\cap X_{g-2,1})}{\area(B\cap X_{g-2,1})}\right\}\\
\geq &h(X_{g-2,1})\succ 1.
\end{aligned}
\ene
Combining \eqref{hxg-3}, \eqref{hxg-4-1},  \eqref{hxg-4-2} and \eqref{hxg-4-3}, for Case-3 we always have
\be\label{hxg-3-t}
\frac{\ell(\Gamma)}{\area(A)}\succ 1.
\ene

Then the conclusion follows from \eqref{hxg-1}, \eqref{hxg-2} and \eqref{hxg-3-t}.
\ep

\begin{rem*}
For a general index $k\in[1,2g-3]$ independent of $g$, through replacing $S_{2,1}$ with a compact hyperbolic surface $S_{k,1}$ of genus $k$ with one geodesic boundary of length $1$, then it follows from the construction and similar argument above that there exists a closed hyperbolic surface $\mathcal{Y}_g$ of genus $g$ such that $$\lambda_k(\mathcal{Y}_g)\asymp \sL_k(\mathcal{Y}_g)\asymp 1.$$
\end{rem*} 

\bp[Proof of Theorem \ref{mt-1}]
It clearly follows from Theorem \ref{mt-1-lb} and Theorem \ref{mt-1-up}.
\ep

\section{Bounds on $\frac{\lambda_2(X_g)}{\sL_2(X_g)}$ for random surfaces of large genus}\label{geometryquanityrandomsurface}

Let $\sM_g$ be the moduli space of closed Riemann surfaces of genus $g$ endowed with the Weil-Petersson metric that has a natural and magic form in Fenchel-Nielsen coordinates due to Wolpert \cite{Wolpert82}. In this section we study the asymptotic behavior of the ratio $\frac{\lambda_2(X_g)}{\sL_2(X_g)}$ over $\sM_g$ for large genus and complete the proof of Theorem \ref{mt-3}. We use the same notations as in \cite{NWX20, WX21}.

Mirzakhani in \cite{mirzakhani2013growth} showed that 
\[\lim \limits_{g\to \infty}\Prob\left(X_g\in \sM_g; \  \lambda_1(X_g)\geq \frac{1}{4}\left(\frac{\ln(2)}{2\pi+\ln(2)} \right)^2 \right)=1.\]
Joint with Xue, the second named author in \cite{WX21} showed that for any $\epsilon>0$,
\[\lim \limits_{g\to \infty}\Prob\left(X_g\in \sM_g; \  \lambda_1(X_g)\geq \frac{3}{16}-\epsilon \right)=1,\]
for which one may see an independent proof by Lipnowski-Wright in \cite{LW21}. One may also see related results for random covers of a fixed closed hyperbolic surface in \cite{MNP20} by Magee, Naud and Puder. Then by \eqref{1/4-ub-ei} we have
\be \label{whp-2-1}
\lim \limits_{g\to \infty}\Prob\left(X_g\in \sM_g; \  \lambda_2(X_g)\asymp 1 \right)=1.
\ene

\noindent For $\sL_2(X_g)$, first by Proposition \ref{boundlk} we know that for all $g\geq 3$,
\be\label{up-l2}
\sup\limits_{X_g\in \sM_g}\sL_2(X_g)\prec \ln g.
\ene

\noindent To prove Theorem \ref{mt-3}, it suffices to show that
\[\lim \limits_{g\to \infty}\Prob\left(X_g\in \sM_g; \  \sL_2(X_g)\succ \ln g \right)=1.\]

Let  $w:\{2,3,\cdots\}\to \mathbb{R}^+$ be a function satisfying 
\be\label{eq-w-l}
\lim \limits_{g\to \infty}w(g)=\infty \ \textit{and} \ \lim \limits_{g\to \infty}\frac{w(g)}{\ln\ln g}=0.
\ene
We prove 
 \begin{proposition}\label{L2random}
The the following  limit holds: \begin{equation*}
   \lim\limits_{g\to \infty} \Prob\Big( X_g\in \M_{g}; \ \sL_2(X_g)\geq 4\ln g-10\ln \ln g-\omega(g)\Big)=1.
\end{equation*}
\end{proposition}

\begin{rem*}
The proof of Proposition \ref{L2random} below actually also yields that for any $\epsilon>0$ and any fixed $k\geq 1$ independent of $g$,
 \[\lim\limits_{g\to \infty} \Prob\Big( X_g\in \M_{g}; \ \sL_k(X_g)\geq (2k-\epsilon)\ln g\Big)=1,\]
which we leave to interested readers. For $k=1$, this was proved in \cite{NWX20}.
\end{rem*}

Let $L>0$ be a constant which may depend on $g$. For $X_g\in \M_{g}$ with $\sL_2(X_g)\leq L$ realized by  $\gamma=\cup_{i=1}^k\gamma_i$ that separates $X$ into three components $X_{g_1,k_{12}+k_{13}},X_{g_2,k_{21}+k_{23}},X_{g_3,k_{31}+k_{32}}$, denoted by $X_1,X_2,X_3$ respectively with $|\chi(X_1)|\leq|\chi(X_2)|\leq |\chi(X_3)|$ for convenience, then the indices here satisfy  \begin{equation}\label{condition}
\left\{
\begin{array}{ll}
  & k_{ij}=k_{ji} \text{ for all } i\neq j,\\
  & \text{At least two } k_{ij}'s \text{ are not } 0,\\
  &|\chi(X_1)|\leq|\chi(X_2)|\leq |\chi(X_3)|,\\
  & \sum_{i=1}^3g_i+\sum_{i<j}k_{ij}=g+2. \\
   \end{array}
   \right.
\end{equation}
Take $\alpha$ to be a  set of simple multi-curve which topologically separates $S_g$ into such three components, and set $$
N_{g_1,g_2,g_3}^{k_{12},k_{13},k_{23}}(X_g,L)=N_\alpha(X_g,L)=\#\{\gamma\in {\rm Mod}_\emph{g} \cdot \alpha,\ell_\gamma(X_g)\leq L\}
$$
as a function on $\M_g$. Now set
\be
L_g=4\log g-10\log\log g-\omega(g).
\ene  
To prove Proposition \ref{L2random},
it suffices to show the following limit $$
\lim\limits_{g\to\infty}\sum\limits_{(\ref{condition})}\E\left(N_{g_1,g_2,g_3}^{k_{12},k_{13},k_{23}}(X_g,L_g)\right)=0,
$$
where
\[\E\left(N_{g_1,g_2,g_3}^{k_{12},k_{13},k_{23}}(X_g,L_g)\right)=\frac{\int_{\sM_g}N_{g_1,g_2,g_3}^{k_{12},k_{13},k_{23}}(X_g,L_g)dX_g}{\Vol(\sM_g)},\]
since
\[ \Prob\Big( X_g\in \M_{g};\ \sL_2(X_g)\leq L_g\Big)\leq \sum\limits_{(\ref{condition})}\E\left(N_{g_1,g_2,g_3}^{k_{12},k_{13},k_{23}}(X_g,L_g)\right).\]

To simplify notation, we write
$$
\left\{
\begin{array}{ll}
  &V^1(\hat{x},\hat{y})=V_{g_1,k_{12}+k_{13}}(x_1,\cdots,x_{k_{12}},y_1,\cdots,y_{k_{13}}),\\
  &V^2(\hat{x},\hat{z})=V_{g_2,k_{21}+k_{23}}(x_1,\cdots,x_{k_{21}},z_1,\cdots,z_{k_{23}}),\\
  &V^3(\hat{y},\hat{z})=V_{g_3,k_{31}+k_{32}}(y_1,\cdots,y_{k_{31}},z_1,\cdots,z_{k_{32}}), \\
  &V^i=V^i(0\cdots,0).\\
   \end{array}
   \right.
$$
We split the proof of Proposition \ref{L2random} into the following three lemmas. 
\begin{lemma}\label{ineqx2bounded}
 For $|\chi(X_2)|\leq 8$ and $L>1$, we have $$
\E\Big(N_{g_1,g_2,g_3}^{k_{12},k_{13},k_{23}}(X,L)\Big)\prec  \frac{L^{3(|\chi(X_1)|+|\chi(X_2)|)-1}e^{\frac{L}{2}}}{g^{|\chi(X_1)|+|\chi(X_2)|}}.
 $$
 \end{lemma}
\begin{proof}
By \cite[Theorem 1.1]{Mirz07}, the Weil-Petersson volume $V_{g,n}(L_1,\cdots,L_n)$ of moduli space $\M_{g,n}(L_1,L_2,\cdots,L_n)$ is a polynomial of degree $6g-6+2n$. If $|\chi(X_1)|\leq|\chi(X_2)|\leq 8$, then all $k_{ij}$ are bounded. Recall that both $V^1$ and $V^2$ are polynomials of bounded degrees. Then  by Mirzakhani's integral formula (MIF) (see \cite[Theorem 7.1]{Mirz07}, or \cite[Theorem 6]{WX21}), \cite[Theorem 3.5]{mirzakhani2013growth} and \cite[Lemma 22]{NWX20} we have 
\begin{equation*}
    \begin{aligned}
    &\E\Big(N_{g_1,g_2,g_3}^{k_{12},k_{13},k_{23}}(X,L)\Big)\\
\prec& \frac{1}{V_g}\int_{\mathbb{R}_+^{k_{12}+k_{13}+k_{23}}}1_{[0,L]}\left(\sum\limits_{p=1}^{k_{12}}x_p+\sum\limits_{q=1}^{k_{13}}y_q+\sum\limits_{r=1}^{k_{23}}z_r\right)
    \cdot V^1(\hat{x},\hat{y})V^2(\hat{x},\hat{z})V^3(\hat{y},\hat{z})\\
&\Pi_{p=1}^{k_{12}}x_p\Pi_{q=1}^{k_{13}}y_q\Pi_{r=1}^{k_{23}}z_r
    \cdot dx_1\cdots dx_{k_{12}}dy_1\cdots dy_{k_{13}}dz_1\cdots dz_{k_{23}} \quad (\text{by MIF})\\
    \prec & \frac{1}{V_g}\int_{\sum\limits_{p=1}^{k_{12}}x_p+\sum\limits_{q=1}^{k_{13}}y_q+\sum\limits_{r=1}^{k_{23}}z_r\leq L}L^{6g_1+6g_2-12+4k_{12}+2k_{13}+2k_{23}} \\
&\cdot V^3(\hat{y},\hat{z}) dx_1\cdots dx_{k_{12}}dy_1\cdots dy_{k_{13}}dz_1\cdots dz_{k_{23}} \quad (\text{by \ \cite[Theorem 1.1]{Mirz07}}) \\
    \prec & \frac{1}{V_g}\int_{\sum\limits_{p=1}^{k_{12}}x_p+\sum\limits_{q=1}^{k_{13}}y_q+\sum\limits_{r=1}^{k_{23}}z_r\leq L}L^{6g_1+6g_2-12+4k_{12}+2k_{13}+2k_{23}}\cdot V^3
\cdot\Pi_{q=1}^{k_{13}}\frac{2\sinh \frac{y_q}{2}}{y_q}\\
    \cdot &\Pi_{r=1}^{k_{23}}\frac{2\sinh \frac{z_r}{2}}{z_r}
    \cdot\Pi_{p=1}^{k_{12}}x_p\Pi_{q=1}^{k_{13}}y_q\Pi_{r=1}^{k_{23}}z_r
    \cdot dx_1\cdots dx_{k_{12}}dy_1\cdots dy_{k_{13}}dz_1\cdots dz_{k_{23}} \\
&(\text{by \cite[Lemma 22]{NWX20}})\\
   \prec  &  \frac{V^3}{V_g}\cdot    \int_{\sum\limits_{p=1}^{k_{12}}x_p+\sum\limits_{q=1}^{k_{13}}y_q+\sum\limits_{r=1}^{k_{23}}z_r\leq L}\exp{\frac{1}{2}\left(\sum\limits_{p=1}^{k_{12}}x_p+\sum\limits_{q=1}^{k_{13}}y_q+\sum\limits_{r=1}^{k_{23}}z_r\right))}\\
\cdot& dx_1\cdots dx_{k_{12}}dy_1\cdots dy_{k_{13}}dz_1\cdots dz_{k_{23}}
\cdot L^{6g_1+6g_2-12+4k_{12}+2k_{13}+2k_{23}}\\
\prec& \frac{e^{\frac{L}{2}}}{g^{|\chi(X_1)|+|\chi(X_2)|}}\cdot L^{6g_1+6g_2-12+5k_{12}+3k_{13}+3k_{23}-1}\\
&(\text{by \cite[Theorem 3.5]{mirzakhani2013growth} or Part $(3)$ of \cite[Lemma 19]{NWX20}})\\
\prec& \frac{L^{3(|\chi(X_1)|+|\chi(X_2)|)-k_{12}-1}}{g^{|\chi(X_1)|+|\chi(X_2)|}}\cdot e^{\frac{L}{2}},
    \end{aligned}
\end{equation*}
which proves this lemma.
\end{proof}
\begin{lemma}\label{ineqx1x2bounded}
For any $m_1\leq 9$ and $L>1$, we have $$
\sum\limits_{|\chi(X_1)|=m_1,|\chi(X_2)|\geq 9} \E\Big(N_{g_1,g_2,g_3}^{k_{12},k_{13},k_{23}}(X,L)\Big)\prec \frac{L^{3m_1-1}e^{2L}}{g^{9+m_1}}.
 $$
\end{lemma}
\bp
Using the same argument as above, by  Mirzakhani's integral formula \cite[Theorem 7.1]{Mirz07} and \cite[Lemma 22]{NWX20},
we have for $L>1$,
\begin{equation}
\begin{aligned}\label{ab-9-1}
&\E\Big(N_{g_1,g_2,g_3}^{k_{12},k_{13},k_{23}}(X,L)\Big)
\prec \frac{L^{6g_1-6+2k_{12}+2k_{13}}}{k_{12}!k_{13}!k_{23}!}\frac{V^2V^3}{V_g}\\
\cdot&\int_{\sum\limits_{p=1}^{k_{12}}x_p+\sum\limits_{q=1}^{k_{13}}y_q+\sum\limits_{r=1}^{k_{23}}z_r\leq L}\exp{\frac{1}{2}\left(\sum\limits_{p=1}^{k_{12}}x_p+\sum\limits_{q=1}^{k_{13}}y_q+2\sum\limits_{r=1}^{k_{23}}z_r\right)}\\
\cdot& dx_1\cdots dx_{k_{12}}dy_1\cdots dy_{k_{13}}dz_1\cdots dz_{k_{23}}
\\
\prec& \frac{L^{6g_1-6+2k_{12}+2k_{13}}}{k_{12}!k_{13}!k_{23}!}\frac{V^2V^3}{V_g}\frac{L^{\sum\limits_{i<j} k_{ij}-1}}{(\sum\limits_{i<j} k_{ij})!}e^L.\\
\end{aligned}    
\end{equation}  
 Use $(*)$ to represent the condition $|\chi(X_1)|=m_1, |\chi(X_2)|\geq 9$ and (\ref{condition}).
Then we have \begin{equation*}
\begin{aligned}
&\sum\limits_{(*)} \E(N_{g_1,g_2,g_3}^{k_{12},k_{13},k_{23}}(X,L))=\sum_{k=2}^\infty\sum\limits_{(*),\sum\limits_{i<j} k_{ij}=k} \E(N_{g_1,g_2,g_3}^{k_{12},k_{13},k_{23}}(X,L))\\
\prec &  \sum_{k=2}^\infty    \frac{L^{k-1}e^LL^{3|\chi(X_1)|}}{k!}  \sum_{\sum k_{ij}=k}      \frac{1}{k_{12}!k_{13}!k_{23}!}
  \sum_{g_i:(*)}\frac{V^2V^3}{V_g} \quad (\text{by \eqref{ab-9-1}})\\
  \prec & \sum_{k=2}^\infty    \frac{L^{k-1}e^LL^{3|\chi(X_1)|}}{k!}  \sum_{\sum k_{ij}=k}      \frac{1}{k_{12}!k_{13}!k_{23}!}
  \sum_{g_i:(*)}\frac{W_{|\chi(X_2)|}W_{|\chi(X_3)|}}{V_g}
\\
& (\text{by \cite[Lemma 3.2]{mirzakhani2013growth} or Part $(2)$ of \cite[Lemma 19]{NWX20}})  \\
  \prec &\sum_{k=2}^\infty    \frac{L^{k-1}e^LL^{3|\chi(X_1)|}}{k!} \cdot \frac{3^k}{k!}\frac{W_{|\chi(X_2)|+|\chi(X_3)|}}{g^{9}V_g} \quad (\text{by \cite[Lemma 23]{NWX20}})\\
\prec &\frac{e^{2L}L^{3m_1-1}}{g^{9+m_1}} \ (\text{by \cite[Theorem 3.5]{mirzakhani2013growth} or Part $(3)$ of \cite[Lemma 19]{NWX20}}), \\
\end{aligned}    
\end{equation*} 
which proves this lemma.
\ep
\begin{lemma}\label{ineqx1bounded}
For $L>1$, we have $$
\sum\limits_{10\leq |\chi(X_1)|} \E\Big(N_{g_1,g_2,g_3}^{k_{12},k_{13},k_{23}}(X,L)\Big)\prec  \frac{e^{2L}}{g^{10}L}.
 $$
\end{lemma}
\begin{proof}
Using the similar argument as above, by  Mirzakhani's integral formula \cite[Theorem 7.1]{Mirz07} we have \begin{equation*}
    \begin{aligned}
    &\sum\limits_{10\leq |\chi(X_1)|}\E\Big(N_{g_1,g_2,g_3}^{k_{12},k_{13},k_{23}}(X,L)\Big)\leq \sum\limits_{10\leq |\chi(X_1)|}\frac{V^1V^2V^3}{V_g}\\
    \cdot &\frac{1}{k_{12}!k_{13}!k_{23}!}\int_{\sum\limits_{p=1}^{k_{12}}x_p+\sum\limits_{q=1}^{k_{13}}y_q+\sum\limits_{r=1}^{k_{23}}z_r\leq L}\exp{\left(\sum\limits_{p=1}^{k_{12}}x_p+\sum\limits_{q=1}^{k_{13}}y_q+\sum\limits_{r=1}^{k_{23}}z_r\right)}\\
\cdot& dx_1\cdots dx_{k_{12}}dy_1\cdots dy_{k_{13}}dz_1\cdots dz_{k_{23}}\\
&(\text{by MIF, \cite[Theorem 1.1]{Mirz07} and \cite[Lemma 22]{NWX20}})\\
\prec & \sum\limits_{10\leq |\chi(X_1)|}  \frac{e^L}{k_{12}!k_{13}!k_{23}!}\frac{W_{|\chi(X_1)|}W_{|\chi(X_2)|}W_{|\chi(X_3)|} }{V_g} \frac{L^{\sum\limits_{i<j}k_{ij}-1}e^L}{(\sum\limits_{i<j}k_{ij})!} \\
& (\text{by \cite[Lemma 3.2]{mirzakhani2013growth} or Part $(2)$ of \cite[Lemma 19]{NWX20}})  \\
\prec& \sum_{k=2}^\infty\sum_{\sum k_{ij}=k}\frac{c^3e^L}{k_{12}!k_{13}!k_{23}!}\frac{L^{k-1}}{k!}\sum_{g_1:10\leq |\chi(X_1)|}  \frac{ W_{|\chi(X_1)|}W_{2g-2-|\chi(X_1)|} }{(2g-2-|\chi(X_1)|)V_g}\\
& (\text{by \cite[Lemma 23]{NWX20}})  \\
\prec &  \sum_{k=2}^\infty e^L\frac{L^{k-1}}{k!}\frac{1}{g^{10}} \quad (\text{by \cite[Lemma 23]{NWX20}})\\
\prec & \frac{e^{2L}}{g^{10}L},
    \end{aligned}
\end{equation*}
which proves this lemma.
\end{proof}

Now we are ready to prove Proposition \ref{L2random}.

\begin{proof}[Proof of Proposition \ref{L2random}]
Take $L_g=4\log g-10\log\log g-\omega(g)$. Then we have \begin{equation*}
    \begin{aligned}
    &\sum\limits_{(\ref{condition})}\E\Big(N_{g_1,g_2,g_3}^{k_{12},k_{13},k_{23}}(X,L_g)\Big)=\sum_{|\chi(X_2)|\leq 8} \E\Big(N_{g_1,g_2,g_3}^{k_{12},k_{13},k_{23}}(X,L_g)\Big)\\
    &+\sum_{m_1=1}^{9}\sum_{|\chi(X_1)|=m_1,|\chi(X_2)|\geq 9} \E\Big(N_{g_1,g_2,g_3}^{k_{12},k_{13},k_{23}}(X,L_g)\Big)
    \\
    &+ \sum_{10\leq |\chi(X_1)|}\E\Big(N_{g_1,g_2,g_3}^{k_{12},k_{13},k_{23}}(X,L_g)\Big).
    \end{aligned}
\end{equation*}
Since $\lim\limits_{g\to \infty}w(g)=\infty$ and $w(g)=o(\ln\ln g)$, by Lemma \ref{ineqx2bounded}, Lemma \ref{ineqx1x2bounded} and Lemma  \ref{ineqx1bounded},
 we have \begin{equation*}
    \sum_{|\chi(X_2)|\leq 8} \E\Big(N_{g_1,g_2,g_3}^{k_{12},k_{13},k_{23}}(X,L_g)\Big)\prec  e^{-\frac{1}{2}\omega(g)},
\end{equation*}
\begin{equation*}
    \sum_{|\chi(X_1)|=m_1,|\chi(X_2)|\geq 9}\E\Big(N_{g_1,g_2,g_3}^{k_{12},k_{13},k_{23}}(X,L_g)\Big)\prec  \frac{e^{-2\omega(g)}}{g^2\log^{18}g}
\end{equation*}
and
\begin{equation*}
    \sum_{10\leq |\chi(X_1)|}\E\Big(N_{g_1,g_2,g_3}^{k_{12},k_{13},k_{23}}(X,L_g)\Big)\prec \frac{e^{-2\omega(g)}}{g\log^{21} g}.
\end{equation*}
So $$
\lim\limits_{g\to\infty}\sum\limits_{(\ref{condition})}\E\Big(N_{g_1,g_2,g_3}^{k_{12},k_{13},k_{23}}(X,L_g)\Big)=0,
$$
which completes the proof.
\end{proof}

\bp[Proof of Theorem \ref{mt-3}]
It clearly follows from \eqref{whp-2-1}, \eqref{up-l2} and Proposition \ref{L2random}.
\ep

\bibliographystyle{amsalpha}
\bibliography{ref}

\end{document}